\documentclass[preprint,12pt]{elsarticle}

\usepackage[utf8]{inputenc}  
\usepackage[T1]{fontenc}
\usepackage{lmodern}
\usepackage[english]{babel}   
\usepackage{microtype}
\usepackage{calc}
\usepackage{amsmath, mathtools}  
\usepackage{amsmath,amsthm,amssymb,euscript} 
\usepackage{hhline}
\usepackage{dsfont} 
\usepackage[top=25mm, bottom=25mm, left=25mm, right=25mm]{geometry}
\usepackage{url}
\usepackage[colorlinks, linkcolor=black, urlcolor=black,citecolor=black]{hyperref}
\usepackage{enumerate, graphicx} 
\usepackage{caption}
\usepackage{subcaption}
\usepackage{algorithmicx} 
\usepackage[section]{algorithm, placeins} 
\usepackage[noend]{algpseudocode}
\usepackage{pgf} 
\usepackage{tikz}
\usepackage{framed}
\usetikzlibrary{positioning,shadows,arrows, shapes, automata, backgrounds, matrix, petri, positioning, patterns, fit, calc, mindmap, trees}
\usepackage{csvsimple}
\usepackage{listings}
\DeclareCaptionFormat{listing}{{\parbox{\textwidth}{#3}}}

\newcommand{\D}[1]{\mathrm{D}_{#1}}						
\newcommand{\Dp}[1]{\frac{\partial}{\partial #1}}		
\newcommand{\Dt}[1]{\frac{\mathrm{d}}{\mathrm{d} #1}}	
\newcommand{\Ds}{\mathcal{D}}							
\newcommand{\Dred}[1]{\tilde{D}_{#1}}					

\newcommand{\R}{\mathbb{R}}				
\newcommand{\diffd}{\mathrm{d}}		    
\newcommand{\Id}{\mathrm{I}}			

\newcommand{\Hessapprox}{\mathcal{B}}	

\newcommand{\Rey}{\operatorname{Re}} 	
\newcommand{\AoA}{\alpha} 				

\newtheorem{theorem}{Theorem}
\newtheorem{definition}{Definition}
\newtheorem{lemma}{Lemma}

\newcommand{\proofofref}{}
\newproof{zproofof}{Proof of \proofofref}

\algdef{SE}[SUBALG]{bindent}{eindent}{}{\algorithmicend\ }%
\algtext*{bindent}
\algtext*{eindent}

\tikzstyle{Base} = [rectangle, draw, fill=blue!20, align=left]
\tikzstyle{Data} = [rectangle, draw, fill=red!20, align=left]
\tikzstyle{Spec} = [rectangle, draw, fill=green!20, align=left]
\tikzstyle{line} = [draw, ->]

\usepackage{lipsum}
\makeatletter
\def\ps@pprintTitle{%
 \let\@oddhead\@empty
 \let\@evenhead\@empty
 \def\@oddfoot{}%
 \let\@evenfoot\@oddfoot}
\makeatother

\begin{document}

\begin{frontmatter}



\title{Combining Sobolev Smoothing with Parameterized Shape Optimization}


\author[inst1]{Thomas Dick\corref{cor1}}
\ead{thomas.dick@scicomp.uni-kl.de}
\author[inst1]{Nicolas R. Gauger}
\author[inst2]{Stephan Schmidt}

\address[inst1]{Chair for Scientific Computing, TU Kaiserslautern, Paul-Ehrlich-Strasse, Building 34, 67663 Kaiserslautern, Germany}
\address[inst2]{Department of Mathematics, Humboldt-Universität zu Berlin, Rudower Chaussee 25, 12489 Berlin, Germany}

\cortext[cor1]{Corresponding author}

\begin{abstract}
On the one hand, Sobolev gradient smoothing can considerably improve the performance of aerodynamic shape optimization and prevent issues with regularity. On the other hand, Sobolev smoothing can also be interpreted as an approximation for the shape Hessian. This paper demonstrates, how Sobolev smoothing, interpreted as a shape Hessian approximation, offers considerable benefits, although the parameterization is smooth in itself already. Such an approach is especially beneficial in the context of simultaneous analysis and design, where we deal with inexact flow and adjoint solutions, also called One Shot optimization. Furthermore, the incorporation of the parameterization allows for direct application to engineering test cases, where shapes are always described by a CAD model. The new methodology presented in this paper is used for reference test cases from aerodynamic shape optimization and performance improvements in comparison to a classical Quasi-Newton scheme are shown.
\end{abstract}


\end{frontmatter}


\section{Introduction}
\label{secintro}

It is well-known from the literature that Sobolev gradient smoothing improves the performance of aerodynamic shape optimization considerably \cite{Jameson94, Pironneau10}, because it counteracts the so-called loss of regularity \cite{Jameson04, Neuberger10}.
This is a phenomenon, where the smoothness of the unknown shape decreases progressively during the optimization. Naturally, using a very large design space, such as all surface mesh vertex positions, which is also called a free node parameterization, is more prone to this problem. There are several interpretations on why this occurs, one is not explicitly demanding the regularity of the search space as a constraint. Hence, using a coarser parameterization, which only produces shapes of sufficient regularity, circumvents this problem~\cite{Zingg02}. \\
The most popular solution to solve this problem is to apply an elliptic partial differential equation to smooth the gradient \cite{Schmidt13, Siebenborn21}. This is equivalent to a reinterpretation of the gradient in a search space of higher regularity. Algorithmically, the resulting Sobolev optimization scheme corresponds to an approximative Newton scheme, where the reduced Hessian is approximated by the respective Sobolev smoothing operator. Indeed, an elliptic partial differential equation often arises as the shape Hessian of many problems, in particular for problems governed by the perimeter \cite{Schmidt18}. This offers another interpretation on why Laplacian smoothing works very well in this context, because high-frequency noise increases the perimeter of the shape considerably. However, a more detailed analysis for a variety of flow problems also shows that elliptic differential equations form a major component of the reduced shape Hessian~\cite{JKusch18, Taasan96, Arian99, Schmidt09}. \\
The interpretation of the smoothing operator as a good approximation of the reduced shape Hessian suggests that this procedure will also be beneficial in the parameterized setting. As such, this paper presents how to combine parameterized aerodynamic shape optimization with shape Hessian approximations, in particular Sobolev smoothing. For this, we present a novel mathematical approach, based on the generalized Faà di Bruno formula \cite{Constantine96, Encinas03}. \\
The mathematical results are used to construct a new approximate Newton scheme with inexact gradients. Within the context of simultaneous analysis and design (SAND), we deal with intermediate flow and adjoint solutions, resulting in approximated values for the objective function and the gradient. This is also called One Shot optimization \cite{Guenther14, Oezkaya10}. The presented methodology, incorporating an analytic approximation of the Hessian, is expected to be superior in comparison to other approximations of the Hessian. This is because it does not suffer from noise when the Hessian is constructed from inexact gradients, as it would happen in Quasi-Newton methods. Overall, this leads to the new One Shot algorithm being superior compared to traditional BFGS-like algorithms. \\
Furthermore, combining shape Hessian approximations with parameterizations, as suggested in this paper, is especially relevant in engineering and industrial application. Here, a parameterization, which is easy to work with in all parts of the design process, i.e., conception, optimization, manufacturing, is required. Computer-aided design (CAD) is the standard representation method for this type of application, with a huge variety of possible, available parameterization methods. Our theoretical results are verbatim valid for such situations. \\
This paper is structured as follows. Section \ref{secadjcalc} recapitulates first-order optimality conditions and summarizes the adjoint calculus. Section \ref{sectionparameter} combines the free node Hessian, i.e., the second derivative w.r.t. the mesh vertex positions, with the Hessian of a parameterized shape optimization problem. This is done in such a way that the formulation is compatible with Sobolev smoothing, which is then introduced in Section \ref{sectionfunctionspaces}. The use of volume and surface smoothing for the parameterized formulation is discussed in Section \ref{secsurfvsvol}. In Section \ref{secoptalg}, we construct a One Shot scheme, which operates in the parameter space, but includes the Sobolev smoothing. We present numerical test cases in Section \ref{secnumeric}. The tests show beneficial convergence behavior of the new One Shot method in comparison to a classical Quasi-Newton scheme. \\
In addition, details about the respective implementations are provided in Appendix \hyperref[secAppendix]{A}.

\section{Adjoint Calculus and Parameterizations}
\label{secadjcalc}

When dealing with flow constrained problems we are usually interested in optimizing the shape of some flow obstacle, e.g., a wing, turbine blade, etc. We call a domain $\Omega \subset \R^{d}$, with boundary $\Gamma$ piecewise of class $\mathcal{C}^{2}(\R^{d},\R)$, the flow domain, see Figure \ref{picflowdomain}. The advantage of this manifold setting is that one could directly apply mathematical theory from shape optimization, without the need for a discretization.

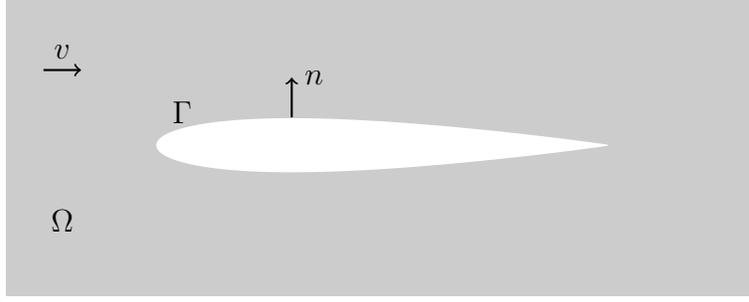
\begin{figure}[ht]
\center
\begin{tikzpicture}
 \node[rectangle, minimum width = 10cm, minimum height = 4cm, fill = black!20, xshift=3cm] {};
 \fill[scale=6, white] plot[smooth] file{Data/n0012.dat} -- cycle node (hole) {};
 \node[] at (-1.25cm, -1cm) {$\Omega$};
 \node[shift = {(0.2, 0.3)}] at (hole.north east) {$\Gamma$};
 \draw[->, thick] (1.8, 0.37) -- (1.8cm, 0.9cm) node[shift = {(0.3, 0)}]{$n$};
 \draw[->, thick] (-1.5cm, 1cm) -- node[midway, above] {$v$} (-1cm, 1cm);
\end{tikzpicture}
 \caption{Flow domain and design shape}
 \label{picflowdomain}
\end{figure}

Consider a vector of design parameters $p \in S_{p} \subset \R^{n_p}$, describing the shape of $\Gamma$, where $S_{p}$ is a set of parameters resulting in feasible designs.
These would normally be the coordinates of a shape parameterization. For example, Hicks-Henne functions \cite{Hicks78}, to deform the surface, or Free-Form Deformation (FFD) points \cite{Sederberg86}, where control points of a regular box around the shape are deformed and the surface coordinates are moved using a smooth spline interpolation. In a broader context, any coordinates of a CAD-based representation of our geometry may be used here.
\begin{definition}[discretized shape optimization problem]
\label{discreteOptProb}
Let $m \in \R^{d \cdot n_m}$ be the coordinates of a $d$-dimensional mesh triangulation of the domain $\Omega$, $F: \R^{n_u} \times \R^{d \cdot n_m} \rightarrow \R$ be a discretized objective function, $M: \R^{n_p} \rightarrow \R^{d \cdot n_m}$ be a parameterization of the computational mesh and $H: \R^{n_u} \times \R^{d \cdot n_m} \rightarrow \R^{n_u}$ a discretization of the flow equations on $\Omega$. Then
 \begin{equation}
 \begin{aligned}
  \min_{p \in S_p} \ & F(u,m) \\
  & m = M(p) \\
  & H(u,m)=0 \\
 \end{aligned}
 \end{equation}
is the discretized shape optimization problem.
\end{definition}
\textbf{Remark:} For the rest of this paper we assume a simplification. The flow equation can be written in a fixed point form $ H(u,m)=0 \Leftrightarrow u=G(u,m)$. For steady-state solutions, this is just a stationary point in the time marching scheme of the used finite volume solver. \\
Now, we can define the Lagrangian function $L : \R^{n_u} \times \R^{n_{\lambda}} \times \R^{d \cdot n_m} \rightarrow \R $,
\begin{equation}
 (u,\lambda,m) \mapsto  F(u,m) + \lambda^{T} \left( G(u,m)-u \right).
\end{equation}
We can now formulate Karush-Kuhn-Tucker (KKT) conditions w.r.t. design parameters, by utilizing the mesh equation $m=M(p)$ \cite[Pages 331-342]{NumOpt99}.
\begin{equation}
\label{KKTcondition}
 \begin{aligned}
 & \D{\lambda} L(u,\lambda,m) = 0  \Leftrightarrow  u =  G(u,m) & \text{(state equation)} \\
 & \D{u} L(u,\lambda,m) = 0  \Leftrightarrow \lambda = (\D{u} G(u,m))^T \lambda + (\D{u} F(u,m))^{T} & \text{(adjoint equation)} \\
 & \D{m} L(u,\lambda,M(p)) \D{p} M(p) = 0   \Leftrightarrow & \text{(design equation)} \\
  & \hspace{0.5cm}  0 = (\D{p} M(p))^{T} \left( (\D{m} G(u,M(p)))^T \lambda + (\D{m} F(u,M(p)))^{T} \right)
\end{aligned}
\end{equation}
Note, that we use the expression $\D{u}$ for the Jacobian matrix of a function with respect to variable $u$. These are well-known first-order optimality conditions and a variety of gradient-based optimization techniques is available to solve those equations. In this paper, a discrete adjoint equation is set up using algorithmic differentiation, to ensure that the adjoint equation is always consistent to the flow and turbulence model on a discrete level \cite{Albring15}. \\
Since we assume the existence and uniqueness of the flow solution, the value of the objective function is completely determined by the design parameters. This means, that it is possible to formulate Definition \ref{discreteOptProb} as an unconstrained optimization problem
 \begin{equation}
  \min_{p \in S_p} \tilde{F}(p),
 \end{equation}
where $\tilde{F}(p) := F(u(M(p)),M(p))$. We use this formulation to calculate expressions for the gradient and the Hessian of the objective w.r.t. the design parameters $p$, in Section \ref{sectionparameter}. It is important to remark that $\tilde{F}(p)$ always implicitly assumes the flow equations to be fulfilled. Additionally, if one assumes the adjoint equation to be fulfilled it is possible to show the following lemma.

\begin{lemma}
\label{lemmaadjcalc}
 Assume that $u,\lambda$ fulfill the flow and adjoint equations, then the following equation holds
 \begin{equation}
  \D{p} \tilde{F}(p) = \D{m} L(u,\lambda,M(p)) \D{p} M(p).
 \end{equation}
\end{lemma}
\begin{proof}
 By definition and application of the chain rule
\begin{equation*}
 \D{m} L(u,\lambda,M(p)) \D{p} M(p) = \left( \lambda^{T} \D{m} G(u,m) + \D{m} F(u,m) \right) \D{p} M(p).
\end{equation*}
Now apply the adjoint calculus, by solving the adjoint equation for $\lambda$ and inserting it
\begin{equation*}
 \lambda^{T} \D{m} G(u,m) + \D{m} F(u,m) = \D{u} F(u,m) (\D{u} \left( G(u,m) - u \right))^{-1} \D{m} G(u,m) + \D{m} F(u,m).
\end{equation*}
Application of the implicit function theorem yields
\begin{align*}
 \D{u} F(u,m) (\D{u} \left( G(u,m) - u \right))^{-T} \D{m} G(u,m) + \D{m} F(u,m) = \D{u} F(u,m) \D{m} u + \D{m} F(u,m).
\end{align*}
Altogether, we get
\begin{equation*}
 \D{m} L(u,\lambda,M(p)) \D{p} M(p) = \left( \D{u} F(u,m) \D{m} u + \D{m} F(u,m) \right) \D{p} M(p) =  \D{p} \tilde{F}(p).
\end{equation*}
\end{proof}

Therefore, the derivative of $\tilde{F}$ w.r.t. $p$ is a reduced derivative via the implicit function theorem.

\section{Hessians and Parameterization}
\label{sectionparameter}

The shape Hessian can be defined either in a continuous or in a discrete setting. In the introduction, we discussed how in engineering applications the geometry is usually not described by a shape $\Gamma$. Instead, a parameterization via a Computer-aided design (CAD) software is used. This allows the designer to apply changes in a more intuitive way, i.e., change the thickness or length of a component instead of moving individual mesh nodes. Also, for industrial applications such CAD-based representations are the common standard, as they allow multiple simulations with varying software tools, e.g., flow, elasticity, or heat transfer simulations, independent of certain mesh file formats. It is important to obtain optimized designs in these terms to ensure that they can be manufactured. For our application, this means we are dealing with a set of design parameters $p \in \mathbb{R}^{n_p}$, as in Definition \ref{discreteOptProb}. If we want to discretize the PDE from Theorem \ref{theoremSobolev}, we have two paradigms.
\begin{enumerate}
 \item The mesh for the CFD computation is our discretization of the shape, implying that we need to formulate a discrete version of the shape Hessian on this level. 
 \item We are interested in sensitivities w.r.t. the parameters $p$ since we rely on optimizing those values for industrial application.
\end{enumerate}

Traditionally, one has worked with the second-order derivatives of the Lagrangian w.r.t. the mesh coordinates as the shape Hessian, i.e., $\D{mm} L(u,\lambda,m)$. We now discuss how to translate this into the situation of a parameterized mesh deformation, where we need to calculate $\D{pp} \tilde{F}(p)$, in such a way that it is compatible with Sobolev smoothing.

\begin{theorem}[discrete parameterization of the shape Hessian]
\label{theomainresult}
 Assume that $\tilde{F}:\R^{n_p} \rightarrow \R, p \mapsto y$ is twice continuously differentiable, for $p \in S_{p} \subset \R^{n_p}$ in the space of possible parameters, then for the second order derivatives of $\tilde{F}$ the following equation holds
 \begin{equation}
 \label{eqmainresult}
  \D{pp} \tilde{F}(p) = \D{p}M(p)^{T} \D{mm} L(u, \lambda, m) \, \D{p}M(p) + \sum_{k=1}^{d \cdot n_{m}} \Dp{m_{k}} L(u, \lambda, m) \, \D{pp} M_{k}(p).
 \end{equation}
Here, $m_{k}$ denotes the $k$-th component of the respective vector $m = M(p)$.
\end{theorem}
\begin{proof}
For the gradient of $\tilde{F}$, Lemma \ref{lemmaadjcalc} gives a connection to the derivatives of $L$. Application of the chain rule to the $i$-th component yields
\begin{equation*}
 \Dt{p_{i}} \tilde{F}(p) = \left( \D{m} L(u, \lambda, M(p)) \D{p} M(p) \right)_{i} = \sum_{k=1}^{d \cdot n_{m}} \Dp{m_{k}} L(u, \lambda, M(p)) \, \Dp{p_{i}} M_{k}(p).
\end{equation*}
Once again, the adjoint calculus with the Lagrangian is crucial to remove dependencies on the flow state $u$ and to transform total into partial derivatives. \\
Now take a look at the $i,j$-th component of the second-order derivative, i.e., the Hessian matrix.
\begin{align*}
 \frac{\diffd^{2}}{\diffd p_{i} \diffd p_{j}} \tilde{F}(p) & = \Dt{p_{i}} \left( \Dt{p_{j}}  \tilde{F}(p) \right) \\
  & = \Dt{p_{i}} \left( \sum_{k=1}^{d \cdot n_{m}} \Dp{m_{k}} L(u, \lambda, M(p)) \, \Dp{p_{j}} M_{k}(p) \right)
\end{align*}
Applying the product rule to each component of this sum results in
\begin{align*}
 \frac{\diffd^{2}}{\diffd p_{i} \diffd p_{j}} \tilde{F}(p) & = \sum_{k=1}^{d \cdot n_{m}} \left( \frac{\diffd \partial}{\diffd p_{i} \partial m_{k}} L(u, \lambda, M(p)) \, \Dp{p_{j}} M_{k}(p) + \Dp{m_{k}} L(u, \lambda, M(p)) \, \frac{\partial^{2}}{\partial p_{i} \partial p_{j}} M_{k}(p) \right)
\end{align*}
For the mixed second order derivative $\frac{\diffd \partial}{\diffd p_{i} \partial m_{k}} L(u, \lambda, m)$ the chain rule can be applied again to replace $p_{i}$ with terms of $m$.
\begin{align*}
 \frac{\diffd^{2}}{\diffd p_{i} \diffd p_{j}} \tilde{F}(p) = & \sum_{k=1}^{d \cdot n_{m}} \left( \sum_{l=1}^{d \cdot n_{m}} \frac{\partial^{2}}{\partial m_{l} \partial m_{k}} L(u, \lambda, M(p)) \, \Dp{p_{i}} M_{l}(p) \Dp{p_{j}} M_{k}(p) \right) \\
 & + \sum_{k=1}^{d \cdot n_{m}} \Dp{m_{k}} L(u, \lambda, M(p)) \, \frac{\partial^{2}}{\partial p_{i} \partial p_{j}} M_{k}(p)
\end{align*}
Reordering the scalar multiplications yields
\begin{align*}
 \frac{\diffd^{2}}{\diffd p_{i} \diffd p_{j}} \tilde{F}(p) = & \sum_{k=1}^{d \cdot n_{m}} \sum_{l=1}^{d \cdot n_{m}} \Dp{p_{i}} M_{l}(p) \frac{\partial^{2}}{\partial m_{l} \partial m_{k}} L(u, \lambda, M(p)) \, \Dp{p_{j}} M_{k}(p) \\
 & + \sum_{k=1}^{d \cdot n_{m}} \Dp{m_{k}} L(u, \lambda, M(p)) \, \frac{\partial^{2}}{\partial p_{i} \partial p_{j}} M_{k}(p).
\end{align*}
The appearing sums are just matrix-vector multiplications. This means that the overall expression is
\begin{align*}
 \D{pp} \tilde{F}(p) & = \D{p} M(p)^{T} \D{mm} L(u, \lambda, M(p)) \, \D{p} M(p) + \sum_{k=1}^{d \cdot n_{m}} \Dp{m_{k}} L(u, \lambda, M(p)) \, \D{pp} M_{k}(p).
\end{align*}
Overall, this concludes the proof.
\end{proof}

The above can be interpreted as a special case of the generalized Faà di Bruno formula~\cite{Constantine96, Encinas03}.

Theorem \ref{theomainresult} connects the second-order derivatives with respect to the mesh $m$ and the parameters $p$. The second term of the right-hand side of Equation \eqref{eqmainresult} can be neglected in most cases, since many parameterizations used in industrial and engineering applications are linear in the design parameters. We will discuss the exact form of these derivatives and their role in more details in Section \ref{secsurfvsvol}. This is a useful feature in a design process, where gradual changes are applied to study the behavior of the system. Naturally, one wants this to be equivalent to one large overall change, i.e., linear in the design parameter. \\
Overall, it is reasonable to approximate the Hessian by the first term of the sum in Equation \eqref{eqmainresult} only. To do a Newton step, all we are missing is the first-order derivative of $\tilde{F}(p)$. It can be expressed simply by multiplication of Jacobian matrices, as seen in Lemma \ref{lemmaadjcalc}, together with the chain rule.
\begin{equation}
 \label{eqfirstordoperator}
  \D{p} \tilde{F}(p) = \D{m} L(u,\lambda,M(p)) \D{p} M(p)
\end{equation}
By using Equations \eqref{eqmainresult} and \eqref{eqfirstordoperator} in a Newton step for the parameterized problem, i.e.,
\begin{equation} 
 \D{pp} \tilde{F}(p) w = -\D{p} \tilde{F}(p)^{T}
\end{equation}
is now transformed into
\begin{equation}
 \D{p}M(p)^T \D{mm} L(u,\lambda,M(p)) \D{p}M(p) w = - \D{p}M(p)^{T} \, \D{m} L(u, \lambda, M(p))^{T}.
\end{equation}
Here $\D{mm} L(u,\lambda,M(p))$ is the reduced, discrete Hessian with respect to the mesh. Together this gives an approximation for the Hessian with respect to the design parameters, thereby eliminating the conflicting objectives described above. Any suitable approximation of the shape Hessian on the mesh $\Hessapprox \approx \D{mm} L(u,\lambda,M(p))$ can now be applied for a Quasi-Newton step $w$ in the parameter space via the formula
\begin{equation}
\label{newtonstep}
 \D{p}M(p)^T \Hessapprox \, \D{p}M(p) w = - \D{p}M(p)^{T} \, \D{m} L(u, \lambda, M(p))^{T}.
\end{equation}

\section{Sobolev smoothing on function spaces}
\label{sectionfunctionspaces}

The idea of Sobolev gradient smoothing is based on the connection between gradients in different function spaces. It is well known from functional analysis, that the gradient as a mathematical object is dependent on the scalar product of the underlying Hilbert space. To illustrate this point further, we will introduce the relevant spaces for our application \cite{Adams03}. \\
Note, that we present the method for scalar-valued functions. When dealing with vector-valued functions, Sobolev smoothing can be generalized naturally by a componentwise application.
\begin{definition}[function spaces]
We define the scalar product
\begin{equation}
 \langle f, g \rangle_{\mathrm{L}^{2}} = \int_{\Omega} \! f(x)g(x) \, \mathrm{d}x
\end{equation}
and call the vector space
\begin{equation}
 \mathrm{L}^{2}(\Omega) := \left\{ f:\Omega\rightarrow \mathbb{R}^d \, \bigg\vert \, \langle f, f \rangle_{\mathrm{L}^{2}}^{\frac{1}{2}} < \infty \right\}
\end{equation}
the space of square-integrable functions. In addition, we define the scalar product
\begin{equation}
 \langle f, g \rangle_{\mathrm{H}^{1}} = \int_{\Omega} \! f(x)g(x) \, \mathrm{d}x + \int_{\Omega} \! \langle \nabla f(x), \nabla g(x) \rangle_{2} \, \mathrm{d}x
\end{equation}
where $\nabla f(x)$ is the weak derivative of $f(x)$ and the vector space
\begin{equation}
 \mathrm{H}^{1}(\Omega) := \left\{ f \in \mathrm{L}^{2}(\Omega) \, \bigg\vert \, \langle f, f \rangle_{\mathrm{H}^{1}}^{\frac{1}{2}} < \infty \right\}
\end{equation}
the Sobolev space of weakly differentiable functions.
\end{definition}
Furthermore, we know that $\mathrm{H}^{1}(\Omega) \subset \mathrm{L}^{2}(\Omega)$ and that the directional derivative in an arbitrary direction $v \in \mathrm{H}^{1}(\Omega)$, defined as
 \begin{equation}
  \Ds f(x; v) := \lim_{\varepsilon \rightarrow 0} \frac{f(x+\varepsilon v) - f(x)}{\varepsilon},
 \end{equation}
is a unique linear mapping on $\mathrm{H}^{1}(\Omega)$. Indeed by Riesz representation theorem, we know that the directional derivative and the gradient are directly linked together.
\begin{theorem}[Riesz representation of the gradient]
\label{corRiesz}
Let $f: \mathrm{H} \rightarrow \mathbb{R}$ be a differentiable function on a Hilbert space $\mathrm{H}$, with inner product $\langle \cdot , \cdot \rangle$, then there exists a unique representative $\nabla f$ such that
\begin{equation}
  \Ds f(x; v) = \langle \nabla f , v \rangle_{\mathrm{H}}.
\end{equation}
We call this representative the gradient of f.
\end{theorem}
One should note that in the case of a finite-dimensional, real-valued vector space, i.e., $\mathbb{R}^{n}$, with the standard scalar product, this leads to the well-known connection between the Jacobian matrix and the gradient.
\begin{equation}
  \D{x}f^{T} = \nabla f
\end{equation}
Furthermore, since the gradient depends on the underlying Hilbert space we can mark this in the notation.
\begin{equation}
 \begin{aligned}
 \nabla^{\mathrm{L}} f : & \ \ \text{gradient in} \ \mathrm{L}^{2}(\Omega) \\
 \nabla^{\mathrm{H}} f : & \ \ \text{gradient in} \ \mathrm{H}^{1}(\Omega)
 \end{aligned}
\end{equation}
The following Theorem \ref{theoremSobolev} demonstrates the connection between the two gradients and introduces the basic equation for Sobolev smoothing.
\begin{theorem}[Sobolev Smoothing]
\label{theoremSobolev}
 For a function $f \in \mathrm{H}^{1}(\Omega)$, which fulfills
 \begin{equation}
  \label{sobolevPDEbound}
   \frac{\partial}{\partial n} \nabla^{\mathrm{H}} f = 0  \ \text{ on } \partial \Omega,
 \end{equation}
 the following equation holds
 \begin{equation}
 \label{sobolevPDE}
  (\mathrm{I} - \triangle) \nabla^{\mathrm{H}} f = \nabla^{\mathrm{L}} f  \ \text{ on } \Omega.
 \end{equation}
\end{theorem}
\begin{proof}
Note that in Equation \eqref{sobolevPDE} the differential operator acts componentwise on the gradient. By Theorem (\ref{corRiesz}), we know that the directional derivative can be expressed in terms of a scalar product with the gradient.
\begin{equation*}
  \Ds f(x; v) = \langle \nabla^{\mathrm{H}} f, v \rangle_{\mathrm{H}} =  \int_{\Omega} \! (\nabla^{\mathrm{H}} f) v \, \mathrm{d}x + \int_{\Omega} \!  \nabla^{\mathrm{L}} (\nabla^{\mathrm{H}} f) \, \nabla^{\mathrm{L}} v \, \mathrm{d}x
\end{equation*}
For the second integral in the equation, we use integration by parts with a zero Neumann boundary as stated in Equation \eqref{sobolevPDEbound}.
\begin{equation*}
  \int_{\Omega} \! (\nabla^{\mathrm{H}} f) v \, \mathrm{d}x + \int_{\Omega} \!  \nabla^{\mathrm{L}} (\nabla^{\mathrm{H}} f) \, \nabla^{\mathrm{L}} v \, \mathrm{d}x =  \int_{\Omega} \! (\nabla^{\mathrm{H}} f) v \, \mathrm{d}x - \int_{\Omega} \!  \triangle(\nabla^{\mathrm{H}} f) \, v \, \mathrm{d}x
\end{equation*}
This can be rewritten using scalar products, resulting in
\begin{equation*}
\int_{\Omega} \! (\nabla^{\mathrm{H}} f) v \, \mathrm{d}x - \int_{\Omega} \!  \triangle(\nabla^{\mathrm{H}} f) \, v \, \mathrm{d}x = \langle \nabla^{\mathrm{H}} f, v \rangle_{\mathrm{L}} - \langle \triangle (\nabla^{\mathrm{H}} f) ,  v \rangle_{\mathrm{L}} = \langle (\mathrm{I} - \triangle) \nabla^{\mathrm{H}} f ,  v \rangle_{\mathrm{L}}.
\end{equation*}
Since the directional derivative is a unique linear mapping on $\mathrm{H}^{1}(\Omega)$ and $\mathrm{L}^{2}(\Omega)$, and $\mathrm{H}^{1}(\Omega) \subset \mathrm{L}^{2}(\Omega)$, by Riesz representation theorem we get the following equation
\begin{equation*}
 \langle \nabla^{\mathrm{L}} f, v \rangle_{\mathrm{L}} = \Ds f(x; v) = \langle \nabla^{\mathrm{H}} f, v \rangle_{\mathrm{H}}.
\end{equation*}
Plugin this into the previous equation we get the following statement
\begin{equation*}
\forall v \in \mathrm{H}^{1}(\Omega) \ : \  \langle (\mathrm{I} - \triangle) \nabla^{\mathrm{H}} f ,  v \rangle_{\mathrm{L}} = \langle \nabla^{\mathrm{L}} f, v \rangle_{\mathrm{L}},
\end{equation*}
which implies the terms are equal outside of sets with measure 0
\begin{equation*}
 \Rightarrow (\mathrm{I} - \triangle) \nabla^{\mathrm{H}} f = \nabla^{\mathrm{L}} f.
\end{equation*}
\end{proof}

The connection to Newton methods becomes apparent if we take a look at the update step for an iterative optimizer. Assume we want to use Newton's method for minimizing $f(x)$, then a single step looks like
\begin{equation}
 x_{k+1} = x_{k} + \delta x \quad \text{with} \quad \D{xx}f(x_k) \delta x = -\nabla^{\mathrm{L}} f(x_k).
\end{equation}
On the other hand, if we do a steepest descent step with the Sobolev gradient the resulting formula is
\begin{equation}
 x_{k+1} = x_{k} + \delta x \quad \text{with} \quad \delta x = - \nabla^{\mathrm{H}} f(x_k).
\end{equation}
Combined with Equation \eqref{sobolevPDE} we get that
\begin{equation}
\begin{aligned}
 \delta x = - (\D{xx}f(x_k))^{-1} \nabla^{\mathrm{L}} f(x_k) \quad \text{for Newton,} \\
 \delta x = - (\mathrm{I} - \triangle)^{-1} \nabla^{\mathrm{L}} f(x_k) \quad \text{for Sobolev.} \\
\end{aligned}
\end{equation}
This motivates an approximation of the Hessian by Sobolev smoothing.

How to proceed depends on whether one considers the motion of the volume vertices as independent, or as dependent on the motion of the boundary vertices only.
Normally, the analytic Hessian only depends on the surface deformation, as movement inside the flow volume does not change the objective function \cite{Schmidt18}. Therefore, a discretized volume approximation will have a rank deficit and be ill-posed. However, when replacing the Hessian with a smoothing operator it seems possible to either smooth the sensitivities utilizing a volumetric operator or to smooth along the surface only. See the next Section \ref{secsurfvsvol} for more details on the surface and volume formulations. \\
Here, we should remark that this also influences the choice of boundary conditions. If we perform the Sobolev smoothing on the design surface $\Gamma$, we can keep the Neumann boundary condition from Theorem \ref{theoremSobolev}. If the Sobolev smoothing is performed on the whole volume mesh, one usually asks for $\nabla^{\mathrm{H}} f = 0$ on the outer freestream boundary, since only the design boundary should be deformed. Notice that this is a stronger assumption than having homogeneous Neumann boundaries everywhere. \\
We demonstrate the formulation of the Sobolev smoothing algorithm on the surface by using the Laplace-Beltrami operator. Assume we have a discretization of this operator on the surface nodes $(\mathrm{I} - \triangle_{\Gamma})$, then we can use it to approximate the reduced Hessian of the Lagrangian function
\begin{equation}
\label{sobolevApprox}
 \D{mm} L(u,\lambda,M(p)) \approx (\mathrm{I} - \triangle_{\Gamma}).
\end{equation}
If we combine this with Equation \eqref{newtonstep}, we get the formulation for Sobolev smoothing on the parameter level
\begin{equation}
\label{smoothingequation}
 \D{p}M(p)^T (\mathrm{I} - \triangle_{\Gamma}) \D{p}M(p) w = - \D{p}M(p)^{T} \, \D{m} L(u, \lambda, M(p))^{T}.
\end{equation}
This formula offers a way to formulate an approximated Newton scheme w.r.t the parameters, even though the design parameters do not necessarily come from a Hilbert space structure. In addition, we can introduce weights $\epsilon_{1}, \epsilon_{2}$ into the equation, to better adapt the operator to a given test case.
\begin{equation}
\label{completesmoothingequation}
 D_{p}M(p)^T (\epsilon_1 \mathrm{I} - \epsilon_2 \triangle_{\Gamma}) D_{p}M(p) \delta p = - \D{p}M(p)^{T} \, \D{m} L(u, \lambda, M(p))^{T}
\end{equation}
For ease of notation, we will denote the entire left-hand side matrix from this equation by
\begin{equation}
\label{systemmatrixequation}
 B := D_{p}M(p)^T (\epsilon_1 \mathrm{I} - \epsilon_2 \triangle_{\Gamma}) D_{p}M(p).
\end{equation}
At this point, we would also like to mention similar smoothing approaches found throughout the literature. Most of them consider applying a partial differential equation to the search direction to change the space in which we perform a descent step and thereby achieve a higher regularity \cite{Schmidt11}. One example is the paper of Kusch, et.al. \cite{JKusch18} on the drag minimization problem for Euler equations, where the authors derive the operator symbol of the Hessian matrix. Another example is the work of Ta'asan \cite{Taasan96}. While this only deals with potential flows and inviscid flows, e.g., situations governed by Euler equations, it offers detailed and deep insight into the theoretical side of the problem. A more recent paper by Müller, Kühl, et.al. \cite{Siebenborn21} suggests solving a p-Laplace problem as a relaxation for the steepest descent step. This is done in a free node optimization setting and can deal effectively with situations where the optimum might contain points or edges in the shape. \\
Finally, we also have to consider the choice of the weights $\epsilon_{1}, \epsilon_{2}$ in Equation \eqref{completesmoothingequation}. One idea is to optimize them as additional parameters in an optimization algorithm, an idea that was proposed by Jameson \cite{Jameson88}. A more in-depth mathematical analysis of the $\epsilon$-values and a strategy to compute them is discussed by Kusch, et.al. in \cite{JKusch16}. All of these approaches are, however, hard to compute in a practical implementation. Therefore, the values are usually determined by a parameter study.

\section{Formulation for Surface and Volume Mesh}
\label{secsurfvsvol}

In the previous Section \ref{sectionfunctionspaces}, we already stated that the Sobolev smoothing method might be formulated on the surface or the volume mesh. Here, we will extend on this point and show some implications of this choice. Recall the smoothing Equation \eqref{newtonstep},
\begin{equation}
\label{completesmoothingequation2}
 D_{p}M(p)^T \Hessapprox D_{p}M(p) \delta p = - \D{p}M(p)^{T} \, \D{m} L(u, \lambda, M(p))^{T}.
 \end{equation}
Depending on the formulation of the Hessian approximation $\Hessapprox$, we might get different formulations. In the case of a volume formulation, assume that the parameterization $M$ consists of two parts. First, the surface parameterization
\begin{equation}
 S: \R^{n_p} \rightarrow \R^{d \cdot n_s} ; p \mapsto s,
\end{equation}
and second, the mesh deformation from surface to volume
\begin{equation}
 V: \R^{d \cdot n_s} \rightarrow \R^{d \cdot n_m} ; s \mapsto m.
\end{equation}
Together, they form the complete design parameterization
\begin{equation}
 M: \R^{n_p} \rightarrow \R^{d \cdot n_m} ; p \mapsto V(S(p)).
\end{equation}
In Equation \ref{eqmainresult}, the Hessian matrix of an individual mesh component w.r.t the parameters $\D{pp} M_{k}$ appears. These component functions $M_{k}$ are combined mappings
\begin{equation}
 M_{k}: \R^{n_p} \rightarrow \R^{d \cdot n_s} \rightarrow \R ; p \mapsto V_{k}(S(p)).
\end{equation}
Therefore, we can apply the Faà di Bruno formula again to get
\begin{equation}
  \D{pp} M_{k}(p) = \D{p}S(p)^{T} \D{ss} V_{k}(s) \, \D{p}S(p) + \sum_{l=1}^{d \cdot n_{s}} \Dp{s_{l}} V_{k}(s) \, \D{pp} S_{l}(p).
 \label{doublefaadibruno}
 \end{equation}
Inserting Equation \eqref{doublefaadibruno} into the original result from Equation \eqref{eqmainresult} yields a formulation for the complete second derivatives,
\begin{equation}
\label{eqfullsmooth}
\begin{aligned}
  \D{pp} \tilde{F}(p) = \ & \D{p}M(p)^{T} \D{mm} L(u, \lambda, m) \, \D{p}M(p) + \\
  & \sum_{k=1}^{d \cdot n_{m}} \Dp{m_{k}} L(u, \lambda, m) \, \D{p}S(p)^{T} \D{ss} V_{k}(s) \, \D{p}S(p) + \\
  & \sum_{k=1}^{d \cdot n_{m}} \Dp{m_{k}} L(u, \lambda, m) \, \left( \sum_{l=1}^{d \cdot n_{s}} \Dp{s_{l}} V_{k}(s) \, \D{pp} S_{l}(p) \right).
\end{aligned}
\end{equation}
In many applications, this formulation can be simplified significantly. Many common surface parameterizations, e.g., the FFD boxes \cite{Sederberg86} or Hicks-Henne functions \cite{Hicks78} used in this paper in Section \ref{secnumeric}, are linear. Therefore, the second order derivative terms of the surface parameterization $\D{pp} S_{l}(p)$ will vanish. This leaves the formulation with some additional cross derivatives for the volume formulation,
\begin{equation}
\sum_{k=1}^{d \cdot n_{m}} \Dp{m_{k}} L(u, \lambda, m) \, \D{p}S(p)^{T} \D{ss} V_{k}(s) \, \D{p}S(p).
\end{equation}
In addition, the mesh deformation between the surface and volume nodes is oftentimes linear as well. For example, the SU2 framework \cite{SU2016} used in this paper computes the deformation with a finite element stiffness approach. See Appendix \hyperref[secAppendix]{A} for more information on the implementation. Numerically, this means assembling a stiffness matrix and solving the corresponding linear system. From this, we can conclude that the second-order derivatives of the mesh deformation $\D{ss} V_{k}(s)$ will vanish as well if an appropriate deformation method is used. \\
In conclusion, we can neglect the second-order derivatives of the mesh parameterization $M(p)$ and formulate a parameterized Sobolev smoothing algorithm based on the first term in Equation \eqref{eqfullsmooth}. For the surface formulation, we end up with
\begin{equation}
  \D{pp} \tilde{F}(p) \approx \D{p}S(p)^{T} (\epsilon_1 \mathrm{I} - \epsilon_2 \triangle_{\Gamma}) \, \D{p}S(p),
\end{equation}
leading to the smoothing equation
\begin{equation}
 D_{p} S(p)^T (\epsilon_1 \mathrm{I} - \epsilon_2 \triangle_{\Gamma}) D_{p} S(p) \delta p = - \D{p}S(p)^{T} \, \D{m} L(u, \lambda, M(p))^{T}.
\end{equation}
Accordingly, we get a similar equation for the volume case,
\begin{equation}
 D_{p} S(p)^T D_{s} V(s)^T  (\epsilon_1 \mathrm{I} - \epsilon_2 \triangle) D_{s} V(s) D_{p} S(p) \delta p = - \D{p}M(p)^{T} \, \D{m} L(u, \lambda, M(p))^{T},
\end{equation}
where $\triangle$ denotes the $d$-dimensional Laplace operator. For more complicated CAD models these linearity assumptions might not hold. In such a case, it would be necessary to compute the additional terms from Equation \eqref{eqfullsmooth} and use them in our approximated Newton step.

\section{Optimization Algorithms}
\label{secoptalg}

\subsection{Reduced SQP Optimizer}
\label{subsectionoptimizer}

In Section \ref{sectionparameter}, we saw how to smooth the gradient for the discrete shape optimization problem from Definition \ref{discreteOptProb}. However, we are still missing one important additional complication. In literally any industrially relevant aerodynamic application we have to deal with additional constraints in the optimization process, other than the mesh and flow equations. In this subsection, we will introduce the handling of additional constraints in the optimization process and demonstrate how our introduced Hessian approximation, by the Sobolev smoothing of gradients, fits naturally into a reduced \textit{sequential quadratic programming} (SQP) framework \cite[Chapter 18] {NumOpt99}. Examples of such constraints could be keeping certain lift or pitching moments for wings, or similar conditions.
\begin{definition}[constrained shape optimization problem]
\label{complOptProb}
Use the terms from Definition \ref{discreteOptProb} and let $E: \mathbb{R}^{n_u} \times \mathbb{R}^{d \cdot n_{m}} \rightarrow \mathbb{R}^{n_E}$ be a set of additional constraints, then we call
\begin{equation}
\begin{aligned}
 \min_{p \in S_p} \ & F(u,m) & \text{(objective function)} \\
 s.t. \ & M(p) = m & \text{(mesh equation)} \\
 & H(u,m) = 0 & \text{(flow equations)} \\
 & E(u,m) = 0 & \text{(equality constraint)}
\end{aligned}
\end{equation}
the constrained shape optimization problem.
\end{definition}
With this we get an extended Lagrangian function $L^{\text{Ext}} : \R^{n_u} \times \R^{n_{\lambda}} \times \R^{n_{\nu}} \times \R^{d \cdot n_m} \rightarrow \R $,
\begin{equation}
\label{Lagrangeextended}
  (u, \lambda, \nu, m) \mapsto F(u,m) + \nu^{T} E(u,m) + \lambda^{T} ( G(u,m)-u ).
\end{equation}
Here, the multiplier $\nu$ weights how much we value improvement in the objective against holding the constraints, and we can derive corresponding KKT conditions.
\begin{equation}
\begin{aligned}
 & \D{\lambda} L^{\text{Ext}}(u, \lambda, \nu, m) = 0 \Leftrightarrow u =  G(u,m) \\
 & \D{u} L^{\text{Ext}}(u, \lambda, \nu, m) = 0  \Leftrightarrow  \lambda =  \D{u} G(u,m)^T \lambda + \D{u} F(u,m)^{T} + \D{u} E(u,m)^{T} \nu \\
 & \D{\nu} L^{\text{Ext}}(u, \lambda, \nu, m) = 0 \Leftrightarrow 0 =  E(u,m) \\
 & \D{m} L^{\text{Ext}}(u, \lambda, \nu, M(p)) \D{p} M(p) = 0 \Leftrightarrow \\
 & \hspace{1cm} 0 =  \D{p} M(p)^{T} \left( \D{m} G(u,M(p))^T \lambda + \D{m} F(u,M(p))^{T} + \D{m} E(u,M(p))^{T} \nu \right)
\end{aligned}
\end{equation}
We want to solve these KKT conditions and a natural way to solve this system of nonlinear equations is to apply a Newton step \cite{phdSchmidt}.
\begin{align}
\label{KKTConstrNewtonStep}
\left[\begin{array}{cccc}
  \D{uu} L^{\text{Ext}} & \D{up} L^{\text{Ext}} & (\D{u} H)^T & (\D{u} E)^T \\
  \D{pu} L^{\text{Ext}} & \D{pp} L^{\text{Ext}} & (\D{p} H)^T & (\D{p} E)^T \\
  \D{u} H & \D{p} H & 0 & 0 \\
  \D{u} E & \D{p} E & 0 & 0 \\
\end{array}\right]
\begin{bmatrix}
  \Delta u \\
  \Delta p \\
  \Delta \lambda \\
  \Delta \nu \\
\end{bmatrix} 
 & = \begin{bmatrix}
 - (\D{u} L^{\text{Ext}})^T  \\
 - (\D{p} L^{\text{Ext}})^T  \\
 -H(u,p) \\
 -E(u,p)
\end{bmatrix}
\end{align}
In this equation, we have the Hessian of the Lagrangian function as the upper left part of the matrix. Also, we still have a dependence on $u$ and $p$. Solving the adjoint problem and applying Lemma \ref{lemmaadjcalc} we can remove the dependency on $u$ and end up with reduced derivatives w.r.t. $p$
\begin{equation}
\label{eqreducedderiv}
 \tilde{D}_{p} F :=  \left( \D{m} F - \D{u} F (\D{u} G - \Id)^{-1} \D{m} G \right) \D{p} M(p),
\end{equation}
and simplify Equation \eqref{KKTConstrNewtonStep} to
\begin{align}
\label{reducedSQPupdate}
\left[\begin{array}{cc}
  B & \tilde{D}_p E^{T} \\
  \tilde{D}_p E & 0 \\
\end{array}\right]
\begin{bmatrix}
  \Delta p \\
  \Delta \nu \\
\end{bmatrix}
=
\begin{bmatrix}
  -\tilde{D} L^{T} \\
  - E \\
\end{bmatrix}.
\end{align}
Here we make use of the Hessian approximation 
\begin{align}
\left[\begin{array}{cc}
  \D{uu} L^{\text{Ext}} & \D{up} L^{\text{Ext}} \\
  \D{pu} L^{\text{Ext}} & \D{pp} L^{\text{Ext}} \\
\end{array}\right]
 & \approx
\left[\begin{array}{cc}
  0 & 0 \\
  0 & B \\
\end{array}\right]
\end{align}
with $B$ defined from Equation \eqref{completesmoothingequation}. \\
The relation to classical SQP methods becomes apparent if we take a look at the following quadratic optimization problem.
\begin{equation}
\begin{aligned}
 \min_{v} & \ \frac{1}{2} v^{T} \tilde{D}_{pp} L^{\text{Ext}} v + \tilde{D}_{p} F v & \text{(objective function)} \\
 s.t. \ & \ \tilde{D}_{p} E v + E = 0 & \text{(equality constraint)}
\end{aligned}
\end{equation}
This quadratic approximation of the Lagrangian combined with linear constraints is known as the basic idea of SQP methods. The quadratic optimization problem has a unique solution $(v, \mu)$, fulfilling the relation
\begin{equation}
\begin{aligned}
 \tilde{D}_{pp} L^{\text{Ext}} v + \tilde{D}_{p} F^{T} + \tilde{D}_{p} E^{T} \mu  &= 0 \\
 \tilde{D}_{p} E v + E & = 0.
\end{aligned}
\end{equation}

\begin{algorithm}[ht]
 \caption{SQP method for equality constraint optimization}
 \label{eqconstSQP}
 \begin{algorithmic}[0]
  \State {\textbf{input} Initial design variables $p_0$, set iteration counter $i=0$}
  \While {err $\geq$ tol}
    \State compute a deformed mesh $ m_i = M(p_i)$
    \State solve the flow equation $u_i=G(u_i,m_i)$
    \State solve the adjoint equations
     \bindent
      \State $ \lambda_i = \D{u} G(u_i,m_i)^T \lambda_i + \D{u} F(u_i,m_i)^T $
      \For {$k=0,1,...,n_{E}$}
       \State $ \lambda_i^{E_{k}} = \D{u} G(u_i,m_i)^T \lambda_i^{E_{k}} + \D{u} E_{k}(u_i,m_i)^T $
      \EndFor
     \eindent
    \State evaluate the design equations and compute reduced gradients
     \bindent
      \State $ \Dred{p} F = \left( \lambda_{i}^{T} \D{m} G(u_{i},x_{i}) + \D{m} F(u_{i},x_{i}) \right) \D{p} M(p)$
      \For {$k=0,1,...,n_{E}$}
       \State $ \Dred{p} E_{k} = \left( (\lambda_{i}^{E_{k}})^{T} \D{m} G(u_{i},m_{i}) + \D{m} E_{k}(u_{i},x_{i}) \right) \D{p} M(p)$
      \EndFor
     \eindent
    \State solve
  	\begin{align*}
\left[\begin{array}{cc}
  B_{i} & \tilde{D}_{p} E^{T} \\
  \tilde{D}_{p} E & 0 \\
\end{array}\right]
\begin{bmatrix}
  v \\
  \nu_{i+1} \\
\end{bmatrix}
=
\begin{bmatrix}
  -\tilde{D}_{p} F^{T} \\
  -E \\
\end{bmatrix}
\end{align*}
   \State $p_{i+1} = p_i + v$
   \State $i=i+1$
  \EndWhile
  \State {\textbf{return} $p_i$ }
 \end{algorithmic}
\end{algorithm}

If we subtract $\tilde{D}_{p} E \mu$ from both sides we have
\begin{align}
\label{reducedSQPupdate2}
\left[\begin{array}{cc}
  \tilde{D}_{pp} L^{\text{Ext}} & \tilde{D}_p E^{T} \\
  \tilde{D}_p E & 0 \\
\end{array}\right]
\begin{bmatrix}
  v \\
  \mu \\
\end{bmatrix}
=
\begin{bmatrix}
  -\tilde{D} F^{T} \\
  - E \\
\end{bmatrix}.
\end{align}
Since we use the approximation $ \tilde{D}_{pp} L^{\text{Ext}} \approx B$ and since the involved matrices are invertible we can identify the solutions of equations \ref{reducedSQPupdate} and \ref{reducedSQPupdate2}. This gives us
\begin{equation}
\begin{aligned}
 \Delta p & = v \\
 \nu + \Delta \nu & = \mu.
\end{aligned}
\end{equation}
At last, all we are missing is a way to calculate $\tilde{D}_{p} E(u,m)$. This gradient can be calculated by the adjoint solver implementation, similar to the gradient of the objective function. Therefore, we need additional calls to the adjoint solver for each constraint and adjoint constraint states $\lambda^{E_{k}}$.

\begin{algorithm}[ht]
 \caption{SQP method for mixed constraint optimization}
 \label{constSQP}
 \begin{algorithmic}[0]
  \State {\textbf{input} Initial design variables $p_0$, set iteration counter $i=0$}
  \While {err $\geq$ tol}
    \State compute a deformed mesh $ m_i = M(p_i)$
    \State solve the flow equation $u_i=G(u_i,m_i)$
    \State solve the adjoint equations
     \bindent
      \State $ \lambda_i = \D{u} G(u_i,m_i)^T \lambda_i + \D{u} F(u_i,m_i)^T $
      \For {$k=0,1,...,n_{E}$}
       \State $ \lambda_i^{E_{k}} = \D{u} G(u_i,m_i)^T \lambda_i^{E_{k}} + \D{u} E_{k}(u_i,m_i)^T $
      \EndFor
      \For {$l=0,1,...,n_{C}$}
       \State $ \lambda_i^{C_{l}} = \D{u} G(u_i,m_i)^T \lambda_i^{C_{l}} + \D{u} C_{l}(u_i,m_i)^T $
      \EndFor
     \eindent
    \State evaluate the design equations and compute reduced gradients
     \bindent
      \State $ \Dred{p} F = \left( \lambda_{i}^{T} \D{m} G(u_{i},x_{i}) + \D{m} F(u_{i},x_{i}) \right) \D{p} M(p)$
      \For {$k=0,1,...,n_{E}$}
       \State $ \Dred{p} E_{k} = \left( (\lambda_{i}^{E_{k}})^{T} \D{m} G(u_{i},m_{i}) + \D{m} E_{k}(u_{i},x_{i}) \right) \D{p} M(p)$
      \EndFor
      \For {$l=0,1,...,n_{C}$}
       \State $ \Dred{p} C_{l} = \left( (\lambda_{i}^{C_{l}})^{T} \D{m} G(u_{i},m_{i}) + \D{m} C_{l}(u_{i},x_{i}) \right) \D{p} M(p)$
      \EndFor
     \eindent
    \State compute a solution $v$ of the quadratic problem
  	       \begin{align*}
            \min_{v} \ \ & \frac{1}{2} v^{T} B v + \tilde{D}_{p} F v  \\
            s.t. \ \ & \tilde{D}_{p} E v + E = 0 \\
            & \tilde{D}_{p} C v + C \geq 0 
           \end{align*}
    \State update the design $p_{i+1} = p_{i}+v$
    \State $i=i+1$
  \EndWhile
  \State {\textbf{return} $p_i$ }
 \end{algorithmic}
\end{algorithm}

To finally summarize our method, we have introduced a parameterized, discrete shape Hessian approximation, associated it with the Sobolev smoothing of free node gradients, shown how the Laplace-Beltrami operator can be used as a preconditioner, and demonstrated that this is the same as using a reduced SQP method with an approximated Hessian. With all those parts in place, Algorithm \ref{eqconstSQP} shows the overall optimization algorithm. \\
As the final step in this Section, we will state what an expansion of the algorithm for inequality constraints looks like. It is entirely possible to incorporate additional inequality constraints as well. We will not give the full analysis here, since it is beyond the scope of this paper. However, since we have an SQP method, we can just use the available adaptation from the literature \cite[Chapter 18]{NumOpt99}. Assume that we have an additional set of constraints $C: \mathbb{R}^{n_u} \times \mathbb{R}^{d \cdot n_m} \rightarrow \mathbb{R}^{n_{C}}$ in Definition \ref{complOptProb}, such that
\begin{equation}
 C(u,m) \geq 0 \quad \text{(inequality constraint),}
\end{equation}
then the resulting optimization method is shown in Algorithm \ref{constSQP}.

\subsection{Hybrid Laplace-Beltrami operator}

Having formulated the optimization algorithm, we have to reconsider the choice of the approximation for $\tilde{D}_{pp} L^{\text{Ext}}(u, \lambda, \nu, M(p)) $. Here, we will discuss two different aspects, not covered until now, to formulate our final expression for $B$. \\
First, when dealing with additional constraints we saw how to extend the Lagrangian and the KKT conditions accordingly
\begin{equation}
 L^{ext}(u,\lambda,\nu,M(p)) =  L(u,\lambda,M(p)) + \nu^{T} E(u,M(p)).
\end{equation}
We propose using the parameterized formulation of Sobolev gradient reinterpretation, as described by Equation \eqref{completesmoothingequation}. Mathematically this means that Equation \eqref{systemmatrixequation} is an approximation of the Hessian of the original Lagrangian without additional constraints.
\begin{equation}
 \tilde{D}_{pp} L^{ext}(u,\lambda,\nu,M(p)) = \tilde{D}_{pp} L(u,\lambda,M(p)) + \nu^{T} \tilde{D}_{pp} E(u,M(p)) \approx \bar{B} + \nu^{T} \tilde{D}_{pp} E(u,M(p))
\end{equation}
Neglecting the Hessian of the constraints will negatively affect the accuracy of our approximation. Furthermore, this is the only dependence of the optimizer on the nonlinearity of the constraint. Numerical experiments have shown that large design updates with this formulation have trouble abiding by the constraints, since the linearization in Equation \eqref{reducedSQPupdate} is no longer valid. Due to this, we propose the introduction of a regularization term
\begin{equation}
 \tilde{D}_{pp} L^{ext}(u,\lambda,\nu,M(p)) \approx \bar{B} + c \Id.
\end{equation}
The constant factor is ideally chosen as $c = \Vert \nu^{T} \tilde{D}_{pp} E(u,M(p)) \Vert$. However, since computing this terms is as computationally expensive as computing the entire Hessian matrix, we need a simpler heuristic to determine $c$ in practice. \\
Second, we have to ask ourselves about the structure and properties of the Laplace-Beltrami operator
\begin{equation}
 \D{p}M(p)^{T} ( \epsilon_1 \mathrm{I} - \epsilon_2 \triangle ) \, \D{p}M(p).
\end{equation}
It can be split into two parts
\begin{enumerate}
 \item The identity $\mathrm{I}$ is the canonical term to regularize the expression and stands for the standard $\mathrm{L}^{2}$ scalar product when thinking about gradient representation.
 \item The Laplace operator $\triangle$ is how higher-order derivative information is introduced into Sobolev smoothing. It has strong smoothing properties onto the sensitivities. In the derivation of the gradient reinterpretation, we showed its connection to the $\mathrm{H}^{1}$ scalar product.
\end{enumerate}
To formulate the Laplace operator we need a spatial structure, i.e., an embedding in $\R^{2}$ or $\R^{3}$. The design parameter space does not have this property, so it must be assembled on the computational mesh and then projected. In contrast to this, the identity can be defined as the componentwise identity for each of the design parameters. Combining both viewpoints we construct the system matrix as follows
\begin{equation}
 \label{hybridlaplacebeltrami}
  B = \D{p}M(p)^{T} ( \epsilon_1 \mathrm{I}_{\Gamma} - \epsilon_2 \triangle_{\Gamma} ) \, \D{p}M(p) + \epsilon_3 \mathrm{I}_{p}.
\end{equation}
This expression is referred to as the \textit{hybrid Laplace-Beltrami operator}. \\
It is motivated by the two considerations above and is a fast and computationally cheap way to regularize our Hessian approximation. We already stated in Section \ref{sectionfunctionspaces} that optimal weights $\epsilon_{1}, \epsilon_{2}$ could be computed analytically. However, the necessary terms are oftentimes too expensive to evaluate in practice. For practical applications of the algorithm, we found that a simple heuristic that ensures that the hybrid Laplace-Beltrami operator \eqref{hybridlaplacebeltrami} remains a symmetric positive definite matrix works well enough.

\FloatBarrier

\subsection{One Shot}
\label{subseconeshot}

The One Shot method is a framework for the simultaneous solution of flow, adjoint, and design equations. Mathematically, it is a Newton-type method to solve the complete nonlinear system of equations in the KKT condition at once \cite{Hamdi11}. We can understand this better when we view it as a natural extension of the adjoint framework. First, start with a coupled fixed point iteration for the flow and adjoint equations, known as the piggyback iteration
\begin{equation}
 \label{eqpiggyback}
\begin{bmatrix}
  u_{i+1} \\
  \lambda_{i+1}
\end{bmatrix}
=
\begin{bmatrix}
 G(u_{i},m) \\
 \D{u} G(u_{i},m)^T \lambda_{i} + \D{u} F(u_{i},m)^T
\end{bmatrix}.
\end{equation}
The following notation can be significantly simplified by introducing the so-called shifted Lagrangian function
\begin{equation}
 N(u,\lambda,x) = L(u,\lambda,x) + \lambda^{T} u =  F(u,x) + \lambda^{T} G(u,x).
 \label{eqshiftedlagrange}
\end{equation}
For the classical One Shot algorithm, the intermediate adjoint solution is directly used to update the mesh coordinates.
\begin{align}
\label{eqoneshotiter1}
\begin{bmatrix}
  u_{i+1} \\
  \lambda_{i+1} \\
  m_{i+1}
\end{bmatrix}
=
\begin{bmatrix}
  G(u_{i},m_{i}) \\
  \D{u} N(u_{i}, \lambda_{i}, m_{i})^T \\
  m_{i} - P_{i} \D{m} N(u_{i}, \lambda_{i}, m_{i})^T
\end{bmatrix}
\end{align}
Such an iteration solves all three KKT conditions at once, hence the name One Shot method. Naturally, the convergence of this coupled iteration is strongly dependent on an appropriate choice of the preconditioners $P_{i}$. Extensive convergence analysis for this formulation has been done in the past. We refer the reader to the works of Hamdi and Griewank \cite{Hamdi11}, where the method is treated as a Quasi-Newton method for the KKT system.

\begin{algorithm}[ht]
\caption{Multistep One Shot with mesh update}
 \begin{algorithmic}[0]
  \State {\textbf{input} Initial values $p_{0}$}
  \For {$i=0,1,\dots,I-1$}
   \State $ m_{i} = M(p_{i}) $
   \For {$j=0,1,\dots,J-1$}
	\State $ u_{j+1} = G(u_{j},m_{i}) $
    \State $\lambda_{j+1} = \D{u} N(u_{j}, \lambda_{j}, m_{i})^T $
   \EndFor
   \State $ \delta p = \D{p} M(p_{i})^{T} \D{m} N(u_{J}, \lambda_{J}, m_{i})^T $
   \State $ p_{i+1} = p_{i} - B_{i} \delta p $
  \EndFor
  \State {\textbf{return} $p_{I}$ }
 \end{algorithmic}
\label{algmultisteponeshot}
\end{algorithm}

\begin{algorithm}[ht]
\caption{constrained multistep One Shot}
 \begin{algorithmic}[0]
  \State {\textbf{input} Initial values $p_{0}$}
  \For {$i=0,1,...,I-1$}
   \State $ m_{i} = M(p_{i}) $
   \For {$j=0,1,...,J-1$}
    \State $ u_{j+1} = G(u_{j}, m_{i}) $
    \State $\lambda_{j+1} = \D{u} N(u_{j}, \lambda_{j}, m_{i})^T$
    \For {$k=0,1,...,n_{E}$}
     \State $\lambda_{j+1}^{E_{k}} = \D{u} N^{E_{k}}(u_{j}, \lambda_{j}^{E_{k}}, m_{i})^T$
    \EndFor
    \For {$l=0,1,...,n_{C}$}
     \State $\lambda_{j+1}^{C_{l}} = \D{u} N^{C_{l}}(u_{j}, \lambda_{j}^{C_{l}}, m_{i})^T$
    \EndFor
   \EndFor
   \State $ \Dred{p} F = \D{m} N(u_{J}, \lambda_{J}, m_{i}) \D{p} M(p_{i})$
   \For {$k=0,1,...,n_{E}$}
    \State $ \Dred{p} E_{k} = \D{m} N^{E_{k}}(u_{J}, \lambda_{J}^{E_{k}}, m_{i}) \D{p} M(p_{i})$
   \EndFor
   \For {$l=0,1,...,n_{C}$}
    \State $ \Dred{p} C_{l} = \D{m} N^{C_{l}}(u_{J}, \lambda_{J}^{C_{l}}, m_{i}) \D{p} M(p_{i})$
   \EndFor
   \State compute preconditioner $ B_{i} $
   \State compute a solution $v$ of the quadratic problem
  	      \begin{align*}
           \min_{v} \ \ & \frac{1}{2} v^{T} B v + \Dred{p} F v  \\
           s.t. \ \ & \Dred{p} E v + E = 0 \\
           & \Dred{p} C v + C \geq 0
          \end{align*}
   \State $ p_{i+1} = p_{i} + v $
  \EndFor
  \State {\textbf{return} $p_{I}$ }
 \end{algorithmic}
 \label{algoneshotdircon}
\end{algorithm}

For practical applications, Equation \eqref{eqoneshotiter1} makes for a poor algorithm, due to a couple of reasons.
\begin{itemize}
 \item Updating the design only after several iterations would intuitively increase the accuracy of the involved functions and gradients. Such algorithms are known as multistep One Shot algorithms.
\end{itemize}
\begin{itemize}
 \item For our purposes the mesh parameterization has to be incorporated in the update procedure as well. In industrial applications updating the geometry can be quite expensive in terms of runtime. The implementation has to project mesh sensitivities back onto the design parameters and additionally call CAD tools, mesh generators, or similar necessary software to compute an updated geometry.
\end{itemize}
Taking the arguments above into consideration the parameterized multistep One Shot algorithm \ref{algmultisteponeshot} can be formulated. However, this method is still lacking a crucial component. For aerodynamic shape optimization, we need a way to incorporate additional constraints into the optimization, as introduced in Subsection \ref{subsectionoptimizer}. There we took the classical approach of computing converged flow and adjoint solutions and then performing a Quasi-Newton step for the design equation, and replaced it with a reduced SQP optimizer. We can combine this with the One Shot algorithm \ref{algmultisteponeshot} by handing intermediate flow and adjoint states to the quadratic subproblem in the SQP algorithm. Combining Algorithms \ref{constSQP} and \ref{algmultisteponeshot} into one method leads to the One Shot optimization algorithm \ref{algoneshotdircon} used in this paper.

\section{Numerical Results}
\label{secnumeric}

\subsection{NACA 0012}
\label{subsecnaca0012}

\begin{figure}[ht]
 \center
 \includegraphics[width=0.5\linewidth]{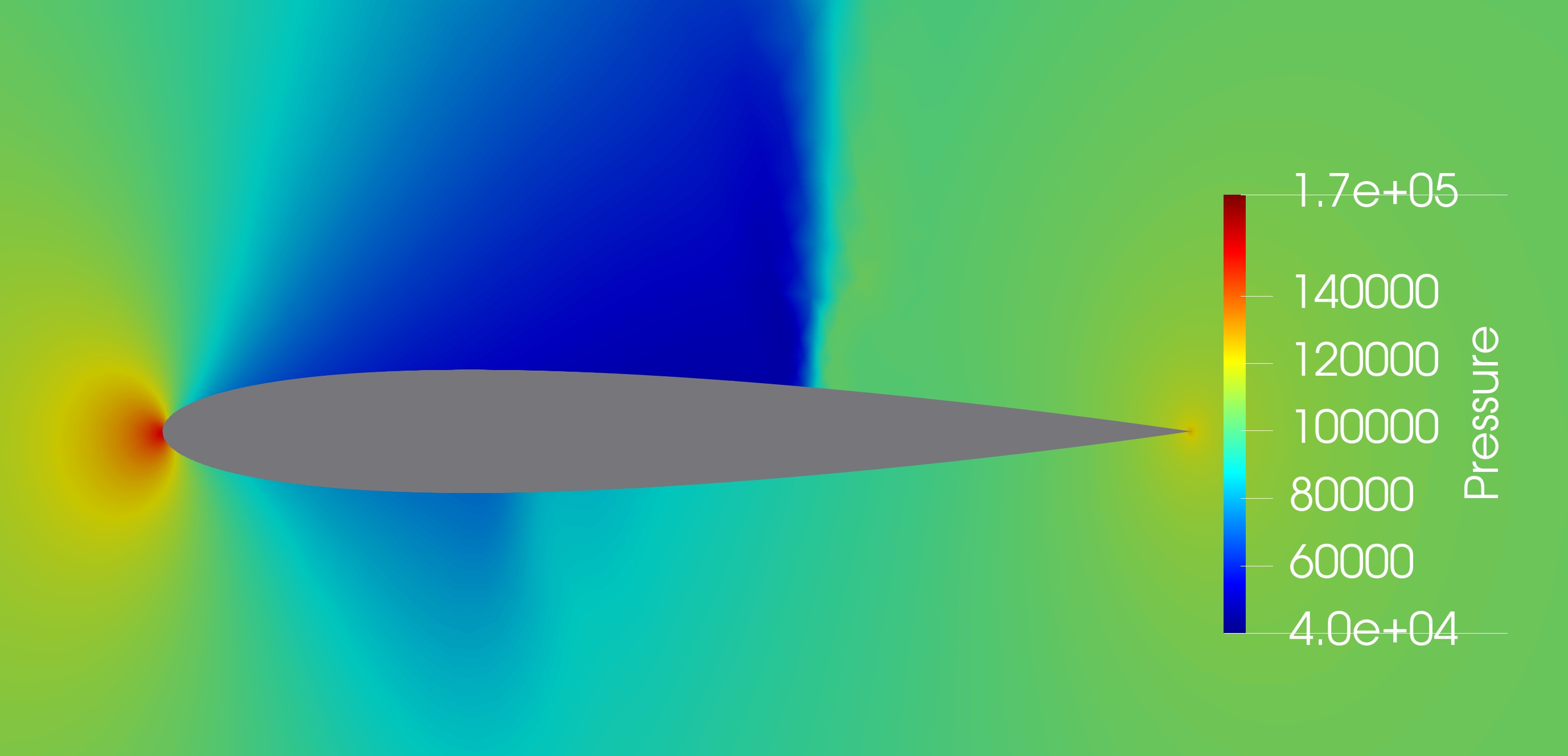}
 \caption{Pressure distribution in the flow field around the NACA 0012 airfoil}
 \label{picNACA0012flow}
\end{figure}

In this section, we will show results for our method applied to standard test cases from aerodynamic shape optimization.
The first test case is the well-known NACA 0012 airfoil. This geometry dates back to the first half of the 20th century \cite{NACA_0012} and is one of the most widespread test cases for CFD in general. For the flow simulation, we use Euler equations, where the conditions are a Mach number of $M=0.8$, an angle of attack $\AoA=1.25^{\circ}$, and standard air, resulting in a typical shock appearing on the upper side of the airfoil. The computational mesh contains $5233$ points, $10216$ elements, and two boundary markers, one for the airfoil and farfield boundary respectively. The resulting flow field can be seen in Figure \ref{picNACA0012flow}. In fact, the drag value in this Euler simulation results from the shock. Therefore, assuming one wants to minimize the drag coefficient it is intuitively clear that the shape optimization must remove the shock. \\
As a first test, we validate the finite elements solver for the Laplace-Beltrami operator by applying it to the surface sensitivities for the drag coefficient $c_{D}$, according to Equation~\eqref{sobolevPDE}. In Figure \ref{picNACA0012surfsens}, the upper part shows the $\mathrm{L}^{2}$ gradient $\nabla^{\mathrm{L}} c_{D}$, while the lower part shows the reinterpreted $\mathrm{H}^{1}$ Sobolev gradient $\nabla^{\mathrm{H}} c_{D}$ on the surface. Both represent the directional derivative for their respective scalar products. As expected, Sobolev reinterpretation shows the classic properties of a Laplace-Beltrami operator, e.g., high sensitivities around the trailing edge are smoothed out.

\begin{figure}[ht]
\center
  \includegraphics[width=0.75\linewidth]{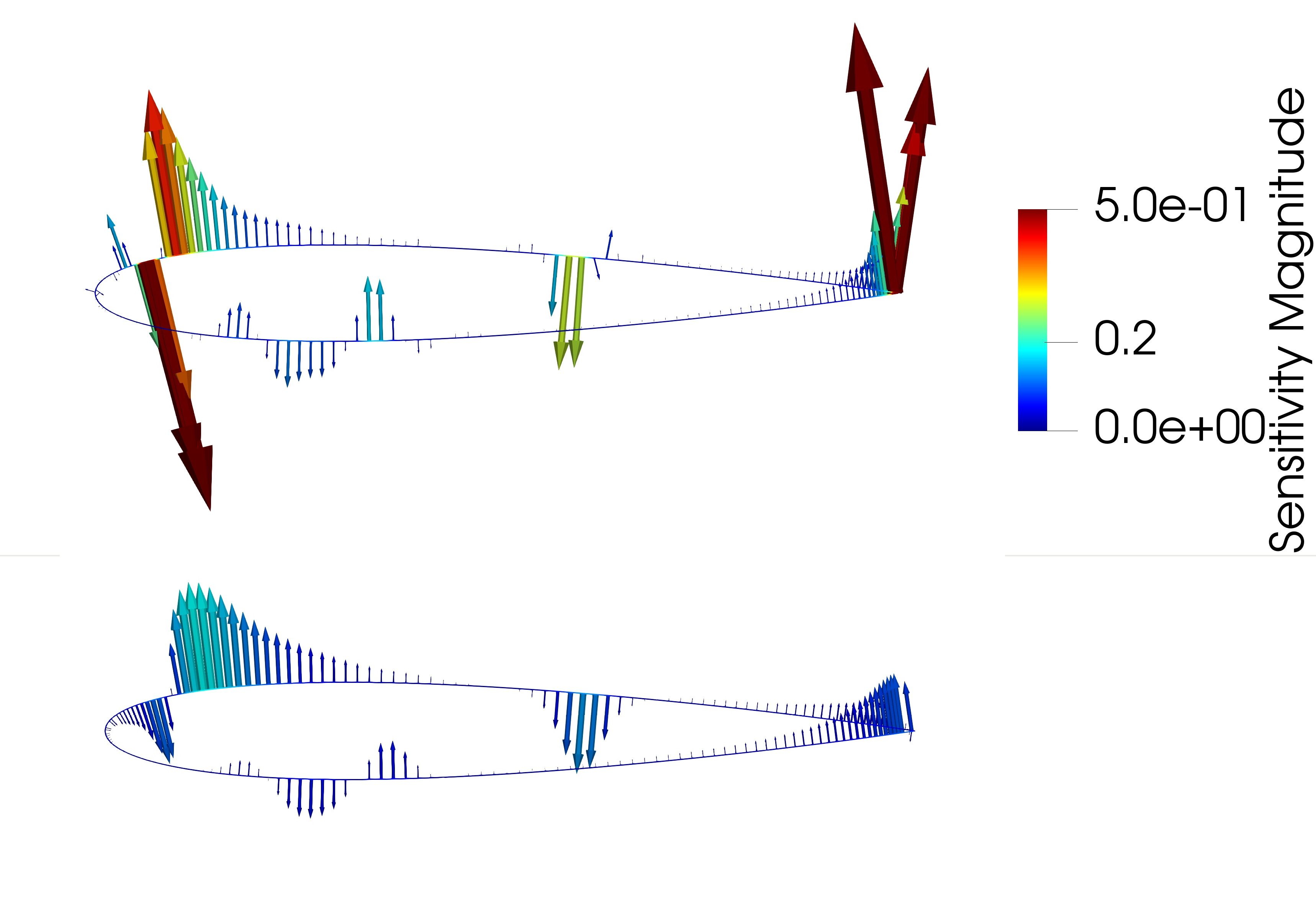}
  \caption{Effect of smoothing on the normal sensitivities on the surface mesh, showing the unaltered $\mathrm{L}^{2}$ gradient (upper) and reinterpreted $\mathrm{H}^{1}$ gradient (lower)}
  \label{picNACA0012surfsens}
\end{figure}

Of course, the main idea of this paper is to apply the Sobolev smoothing method to the design parameters. For this purpose, the airfoil is parameterized by 38 Hicks-Henne bump functions. 19 to deform the upper and 19 to deform the lower curve. The peaks of these functions are distributed equally along the chord at relative lengths of $0.05, 0.10, \dots, 0.95$. \\
For the optimization, we minimize the drag coefficient $c_{D}$, in drag counts ($1$ count $=10^{-4}$), such that the lift coefficient of $32.69$ lift counts ($1$ count $=10^{-2}$) is kept constant. The results of the optimization can be seen in Figure \ref{picOptNACA0012}. We compare the optimization algorithm \ref{eqconstSQP} with our approximated Hessian matrix to a simple constrained gradient descent method, with projection to the constraints. Also, to better understand how the choice of $\epsilon_{1}$ and $\epsilon_{2}$ effects the optimization, we test two different settings. In \cite{JKusch16}, it was suggested to choose a ratio $\epsilon_{1}= 10 \cdot \epsilon_{2}$ between the two parameters. This is test case dependent and a parameter study for the NACA 0012 case resulted in the choice of $\epsilon_{1}= 1.0, \epsilon_{2}=0.0625$. To demonstrate this is a good choice, we show a second setting using $\epsilon_{1}= 1.0, \epsilon_{2}=0.625$. With these settings, the Laplace-Beltrami operator is assembled on the design surface according to Equation \eqref{systemmatrixequation}.

\begin{figure}[ht]
\center
  \includegraphics[width=0.85\linewidth]{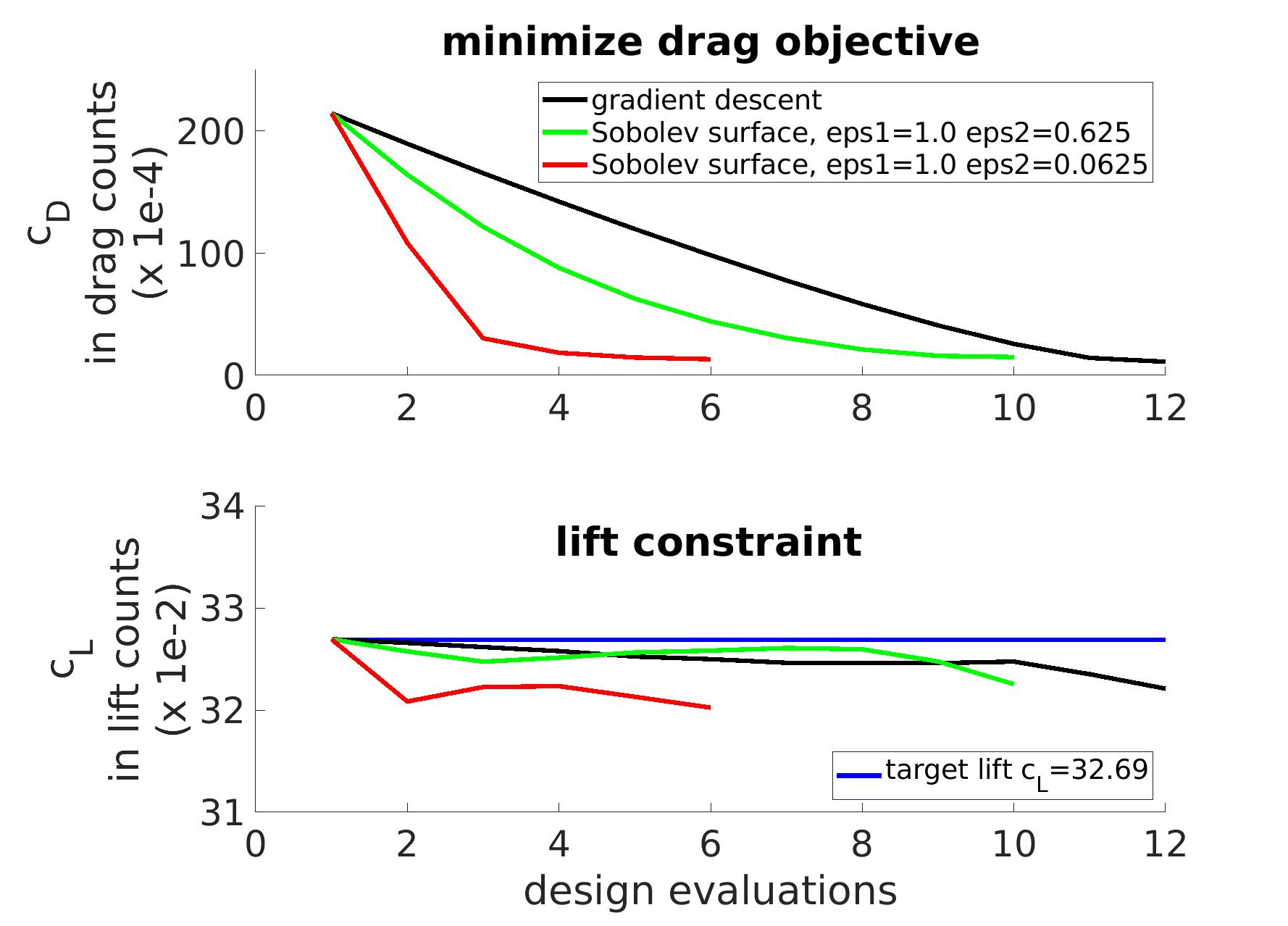}
  \caption{Comparison of the shape optimization for the NACA 0012 test case}
  \label{picOptNACA0012}
\end{figure}


All methods achieve a similar level of improvement in the objective function, with the Sobolev reinterpretation of the gradients yielding a considerably faster convergence than gradient descent. For all methods, the final $c_{D}$ values for the optimized shapes are about $13$ drag counts, which is within the range of spurious drag. The deviation in the lift constraint is more pronounced for the Sobolev methods. However, it stays within $0.6$ lift counts of the target lift, which is within the numerical accuracy. Comparing the two Sobolev variants to each other, the first set of parameters is noticeably faster in the decrease of the objective than the second, at the cost of less than one lift count.

Due to the inviscid nature of this test case, optimal solutions have to be shock free. The computed drag values in those shock free situation stem from numerical stabilization. To this end, Figure \ref{picNACA0012flowcomp} shows the shape and pressure field for the starting geometry (left), the result of our Sobolev smoothed SQP optimization with $\epsilon_{1}= 1.0, \epsilon_{2}=0.0625$ (center), and gradient descent (right). Each optimization algorithm efficiently removes the shock on the upper side of the airfoil and thereby the shock-induced drag. Upon closer inspection, one could argue that a negligible shock remains for the gradient descent algorithm. This is in correspondence to the spurious nature of the drag as discussed above.

\begin{figure}[H]
\center
  \includegraphics[width=\linewidth]{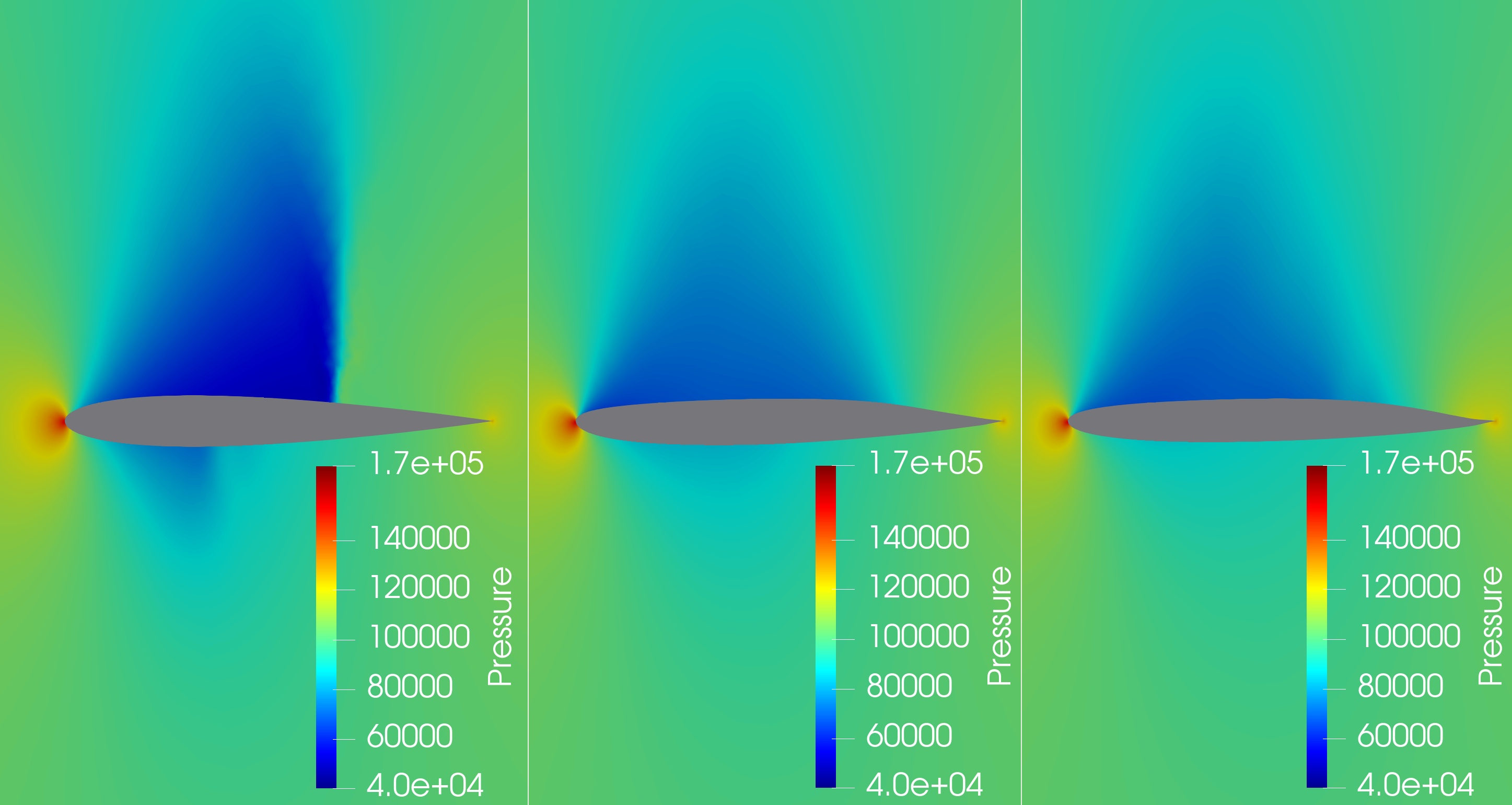}
  \caption{Flow field comparison for original and optimized NACA 0012 airfoil. Original design (left), results using Sobolev smoothing with $\epsilon_{1}= 1.0, \epsilon_{2}=0.0625$ (center) and gradient descent (right)}
  \label{picNACA0012flowcomp}
\end{figure}

\FloatBarrier

\subsection{ONERA M6}

The NACA 0012 test in Section \ref{subsecnaca0012} gave a proof of concept for our algorithm on a well-understood reference test case. However, we are interested in applying preconditioning to more involved optimization problems. To demonstrate the potential for more complex geometries, and especially other parameterizations with higher numbers of design parameters, we take a look at the ONERA M6 test case. This is one of the most common examples of a 3D transonic wing in the CFD community, see AGARD 1979 \cite[Appendix B1]{Agard79}. The standard flow conditions for this case are an angle of attack of $\AoA=3.06^{\circ}$, a Mach number of $M=0.8395$, and a Reynolds number of $\Rey=11.72 \times 10^{6}$, where the Reynolds number is computed from a reference mean chord length $l_C=0.64607m$ and a wing area $A_{ref}=0.7532m^{2}$. Under these conditions, the solution features a transonic flow with a double shock on the upper surface, called a lambda shock. The characteristic pressure distribution on the upper side of the wing is depicted in Figure \ref{picONERAM6origCp}. \\
Wind tunnel data for the pressure coefficients on seven spanwise cross-sections are stated in the reference database \cite{Agard79}. Their positions in percentages of the total length are listed in Table \ref{tablewindtunnelpos} below. For each position 34 experimental measurements at different locations, along the cross-section, are given.
\begin{table}[ht]
\centering
 \begin{tabular}{|c c c c c c c c|}
 \hline
 section & 1 & 2 & 3 & 4 & 5 & 6 & 7 \\
 \hline
 position $\frac{y}{b}$ & 0.20 & 0.44 & 0.65 & 0.80 & 0.90 & 0.96 & 0.99 \\
 \hline
 \end{tabular}
 \caption{Cross-section positions for wind tunnel measurements.}
 \label{tablewindtunnelpos}
\end{table}

\begin{figure}[ht]
\begin{minipage}{0.5\textwidth}

 \center
 \includegraphics[width=\linewidth]{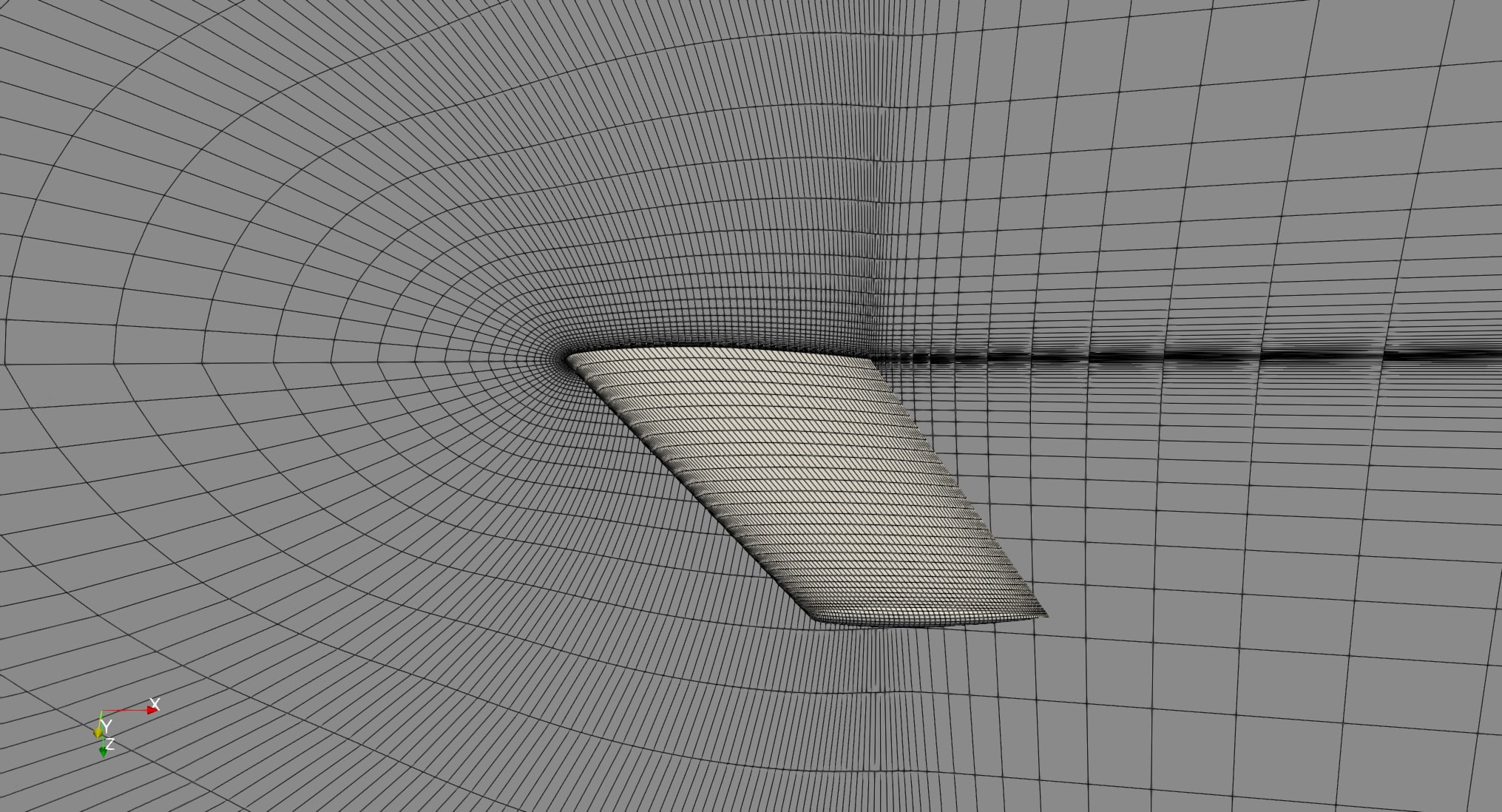}
 \caption{The computational mesh for the ONERA M6 geometry}
 \label{picONERAM6mesh}

 \center
 \includegraphics[width=\linewidth]{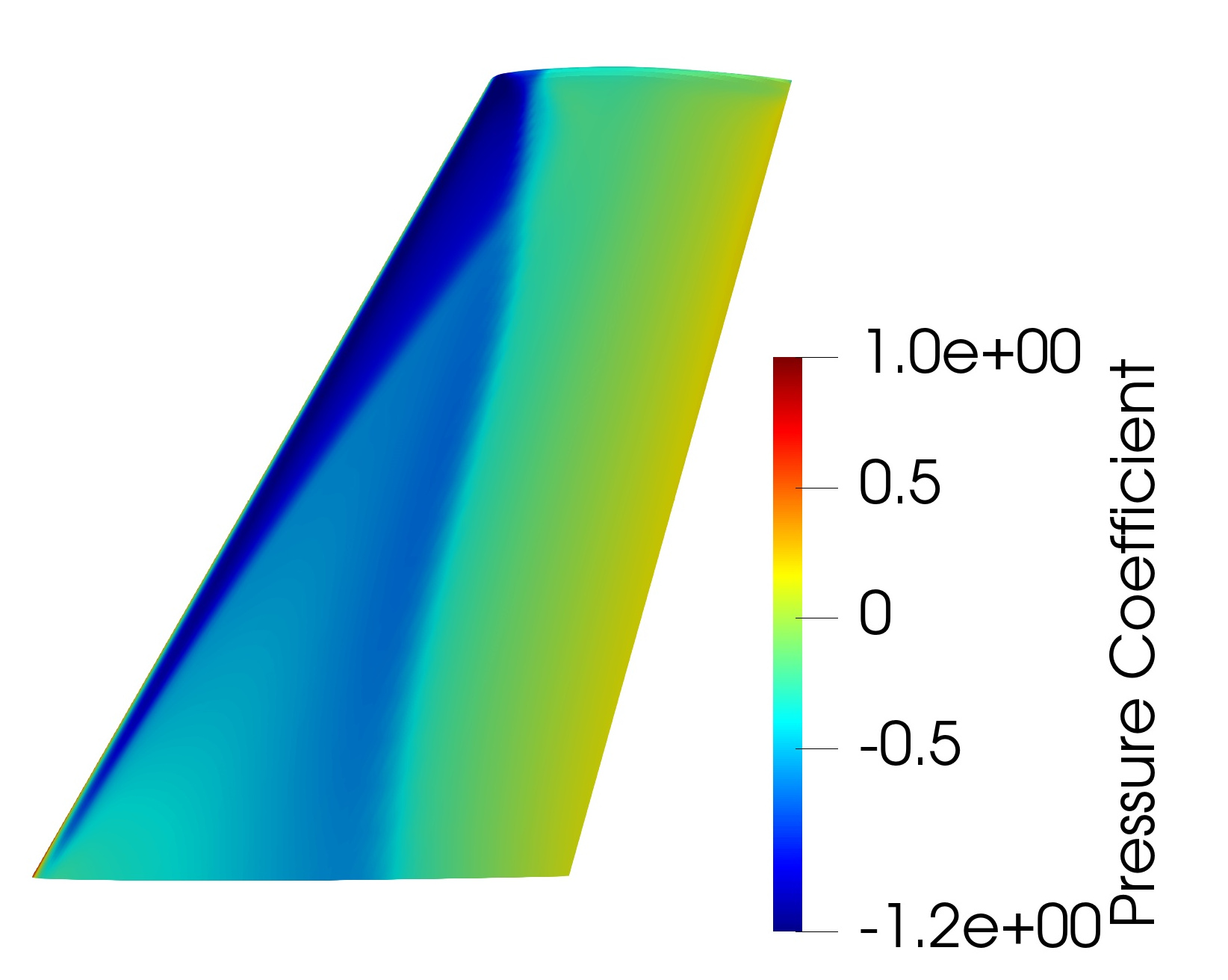}
 \caption{Pressure coefficient for the original ONERA M6 wing}
 \label{picONERAM6origCp}

\end{minipage}
\begin{minipage}{0.5\textwidth}

 \center
 \includegraphics[width=\linewidth]{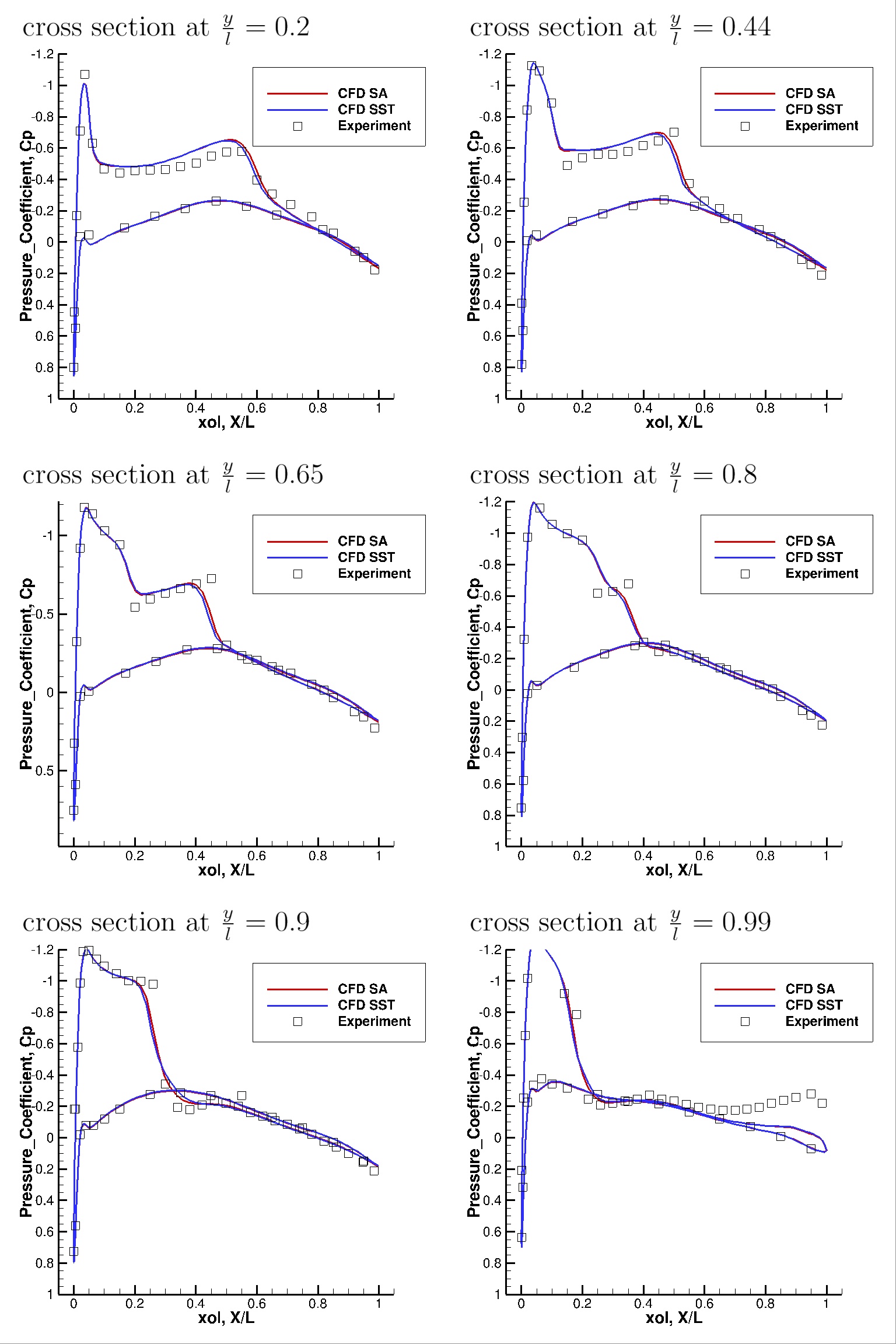}
 \caption{CFD and experimental pressure distribution}
 \label{picONERAM6expvalidation}

\end{minipage}
\end{figure}

For this test, we use a structured C-type mesh for the ONERA M6 wing, consisting of $306,577$ points and $294,912$ rectangular elements. Around the wing the elements follow a logarithmic thickness progression, resolving the boundary layer. The mesh features a farfield, a symmetry plane, and a Navier-Stokes wall for the wing surface as boundaries. \\
For the flow simulation, the Reynolds-averaged Navier–Stokes (RANS) equations with a shear stress transport (SST) turbulence model \cite{Menter94} are solved with a Jameson-Schmidt-Turkel (JST) scheme \cite{JST17}. Artificial dissipation coefficients for this test were set to $d^{(2)}=\frac{1}{2}$ and $d^{(4)}= \frac{1}{128}$. For time integration we use a simple implicit Euler scheme. \\
The parameterization is a three-dimensional FFD box around the wing, which has 10 cells in x-direction, 8 cells in y-direction, and 1 cell in z-direction, with the control points being only allowed to move in z-direction. Other deformations are suppressed to keep the span and chord length of the wing constant. Overall, this gives us 198 design parameters for the optimization. \\
We start with a validation of the mesh and solver settings. A comparison between the CFD solution and experimental wind tunnel data for the original ONERA M6 test case is plotted in Figure \ref{picONERAM6expvalidation}. The CFD simulation can predict the shock position on the upper wing surface well and shows a good correlation to the experimental data. The accuracy is comparable to other relevant numerical studies featuring the same test case \cite{NASA14, NASA02}. For cross-sections at $20\%$, $44\%$, $65\%$, and $90\%$ simulation results correspond well with the experimental data. Furthermore, both the SST and SA turbulence models yield very similar results. Nonetheless, the flow solver has some issues predicting the exact behavior around the merging double shock at $80\%$ of the length and on the wingtip at $99\%$, due to limitations in the turbulence model. More accurate solutions would require the use of a more involved Reynolds stress tensor model, and for reference, we refer the reader to the work of Jakirlić, Eisfeld, et. al. \cite{Eisfeld07}. In total, the test case setting gives us a decent resolution of the physical flow and is a stable basis for shape optimization. \\
For optimization, we minimize the drag coefficient $c_{D}$, with respect to a mix of equality and inequality constraints. The equality constraint keeps the lift coefficient of the original design $c_{L}=24.99$ lift counts and as inequality constraints, the minimal airfoil thickness is kept at the $5$ different cross-sections listed in Table \ref{tableminthickness}.
\begin{table}[ht]
\centering
 \begin{tabular}{|c c c c c c|}
 \hline
 position $\frac{y}{b}$ & 0.0 & 0.2 & 0.4 & 0.6 & 0.8 \\
 \hline
 minimum thickness & 0.077 & 0.072 & 0.066 & 0.060 & 0.054 \\
 \hline
 \end{tabular}
 \caption{Minimum thickness constraints for shape optimization.}
 \label{tableminthickness}
\end{table}
For comparison, we evaluate the optimization Algorithm \ref{constSQP}, with our approximated Hessian matrix, together with two other optimizers. First, a simple constrained gradient descent method with projection to the constraints, which can be obtained by simply using $B=\mathrm{I}$ in Algorithm \ref{constSQP}. Second, we use the SLSQP optimization algorithm as a reference point \cite{Kraft88}. This is the standard Quasi-Newton algorithm within the python SciPy optimize package \cite{2020SciPy-NMeth}. It can be seen as a typical example of an SQP algorithm using gradient-based, iterative updates, e.g., a BFGS formula, to determine an approximated Hessian. \\
For the reduced SQP algorithm, we use two different ways to assemble the hybrid Laplace-Beltrami operator. One on the surface mesh with weights $\epsilon_{1}= 56.9$, $\epsilon_{2}= 0.9$, and $\epsilon_{3}=0.1$. For the other we assemble everything in the volume mesh with weights $\epsilon_{1}= 0.0$, $\epsilon_{2}= 7.1$, and $\epsilon_{3}=0.1$.


\begin{figure}[ht]
 \center
 \includegraphics[width=\linewidth]{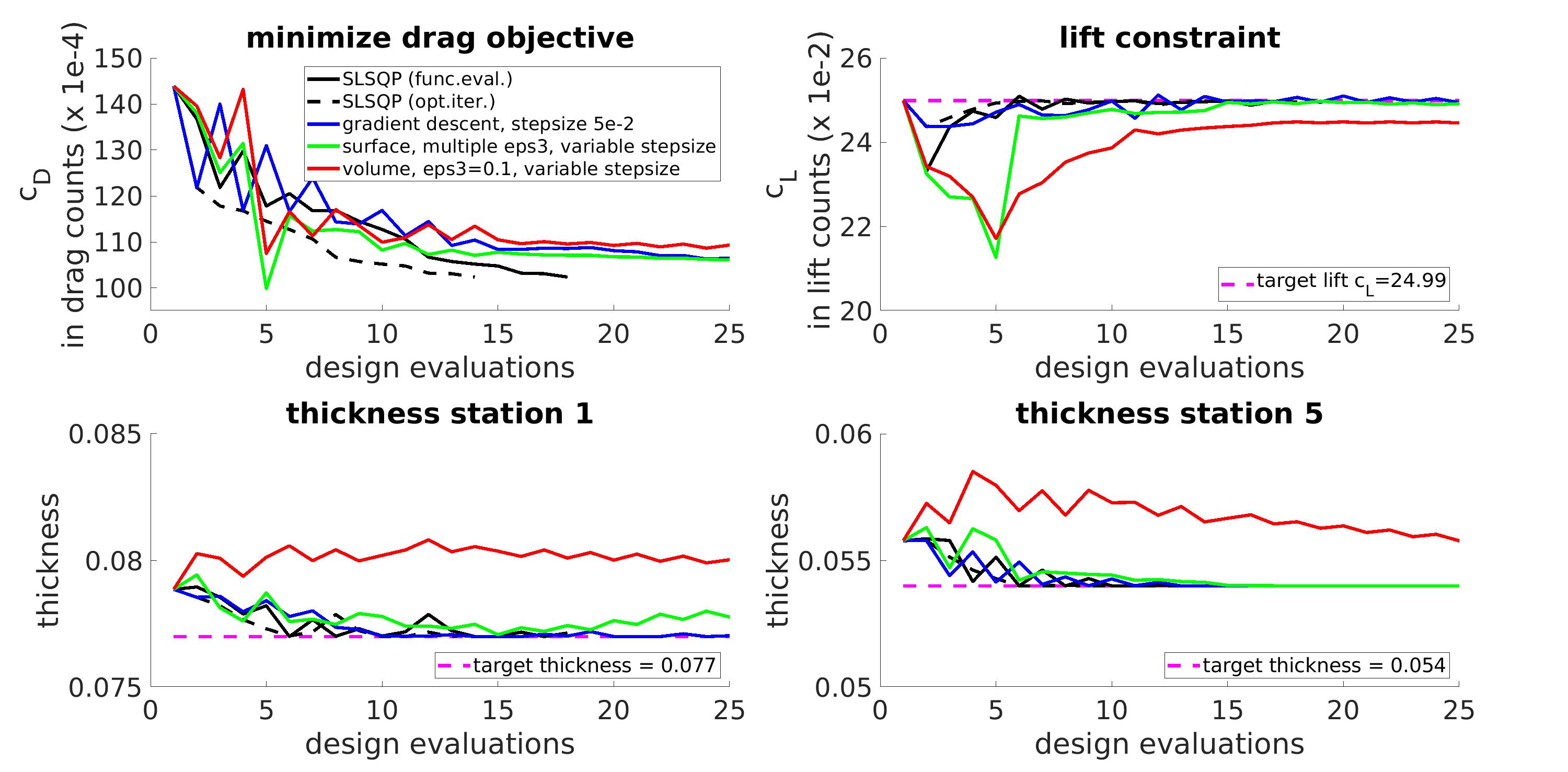}
 \caption{Comparison of different optimizers for the ONERA M6 test case}
 \label{picONERAM6optcomp}
\end{figure}

\begin{table}[ht]
\centering
 \begin{tabular}{| c | c c c c |}
  \hline
   & RSQP surface & RSQP volume & SLSQP & grad. desc. \\
  \hline
   $c_{D}$ (in drag counts) & $105.3$ & $108.1$ & $102.4$ & $105.6$ \\
  \hline
   relative reduction & $26.8 \%$ & $24.89 \%$ & $28.85 \%$ & $26.6 \%$ \\
  \hline
 \end{tabular}
 \caption{$c_{D}$ improvement for different optimization strategies.}
 \label{tableRSQPcomp}
\end{table}

Results of the optimization can be seen in Figure \ref{picONERAM6optcomp}. The objective and the constraints are plotted over the course of the optimization. We can observe that all three methods achieve an improvement in the objective function, from a starting value of $143.9$ drag counts, albeit at different levels. The exact resulting values are shown in Table \ref{tableRSQPcomp}. Using reduced SQP with the smoothing matrix assembled on the surface results in an optimized drag value of $105.3$ drag counts, this represents a $26.8 \%$ decrease in the objective function, while smoothing on the volume mesh results in a comparatively smaller improvement to $108.1$ drag counts, which is a $24.89 \%$ decrease. In comparison, gradient descent results in a drag coefficient of $105.6$ drag counts, i.e., a $ 26.6 \%$ drop. All of these optimizations result in very similar values, while the SLSQP algorithm can achieve an even smaller objective function value of $c_{D}=102.4$ drag counts, which represents a $28.85 \%$ decrease. At the same time, all of the optimization algorithms adhere to the optimization constraints very well. The lift constraint is kept up to a small error margin of less than $2.1 \%$. Geometric inequality constraints are also kept and can be observed to be active constraints in the achieved minimum as well. \\
Furthermore, we can observe how the optimizer oscillates between steps improving the objective function and steps increasing the feasibility, e.g., adherence to the constraints. This is a typical behavior for reduced SQP methods with approximated Hessian matrices, see for example \cite{Gherman07}. They aim at reducing the Lagrangian function which may lead to some steps increasing the objective function to obtain better feasibility in the constraints.

\begin{figure}[ht]
 \center
 \includegraphics[width=\linewidth]{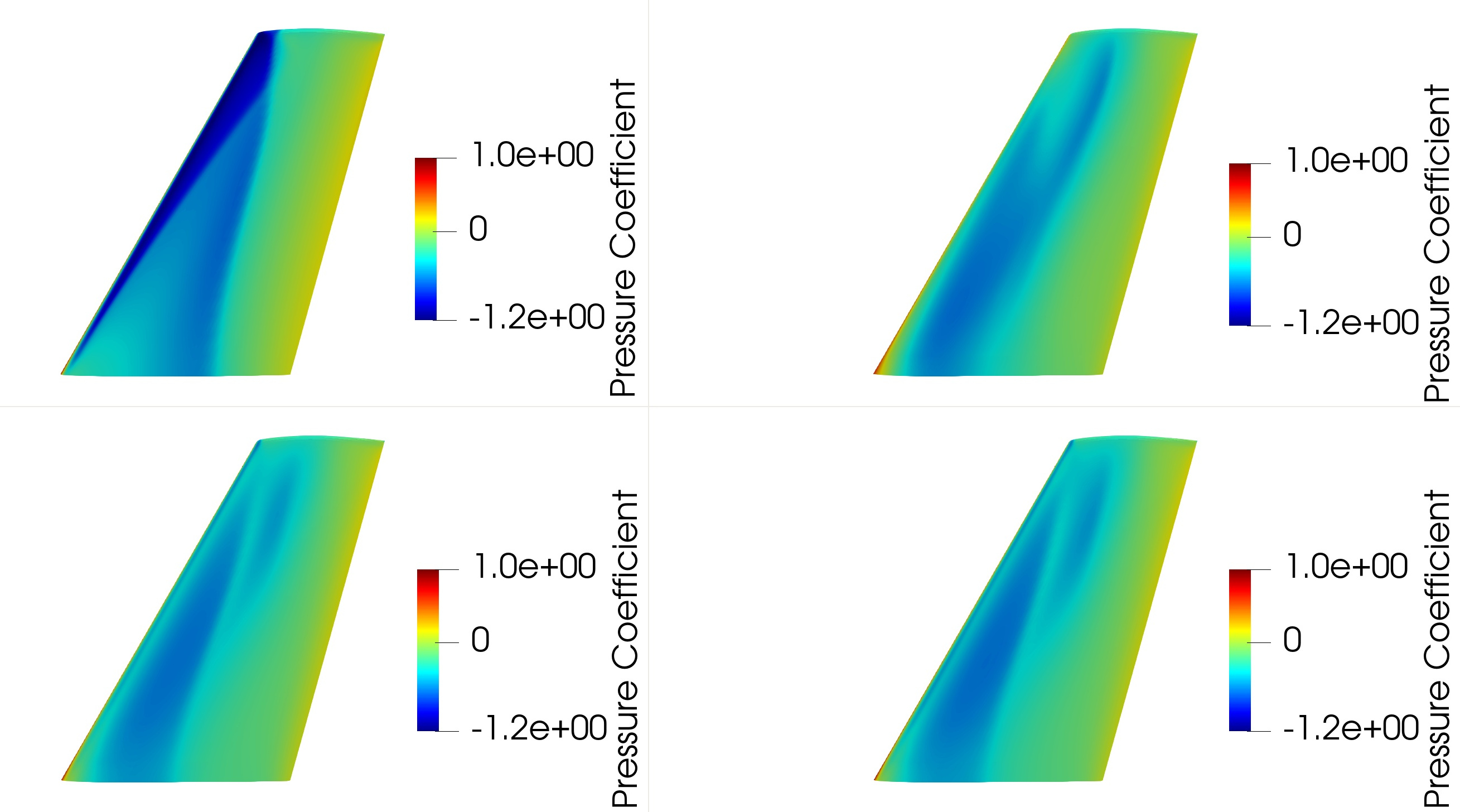}
 \caption{Comparison of optimized flow fields. Original design (upper left), result SLSQP (upper right), result gradient descent (lower left), result Sobolev smoothing (lower right)}
 \label{picONERAM6ffcomp}
\end{figure}

We can take a look at the flow fields around the final designs to achieve a better understanding of the minima found by the different optimization algorithms. A comparison is shown in Figure \ref{picONERAM6ffcomp}. In the upper left, we can see the original wing with the characteristic lambda shock on the surface. All of the used optimizers can remove most of the double shock system. This is important since shock-induced drag is an important source of the overall drag in this test case.

\begin{figure}[ht]
 \center
 \includegraphics[width=\linewidth]{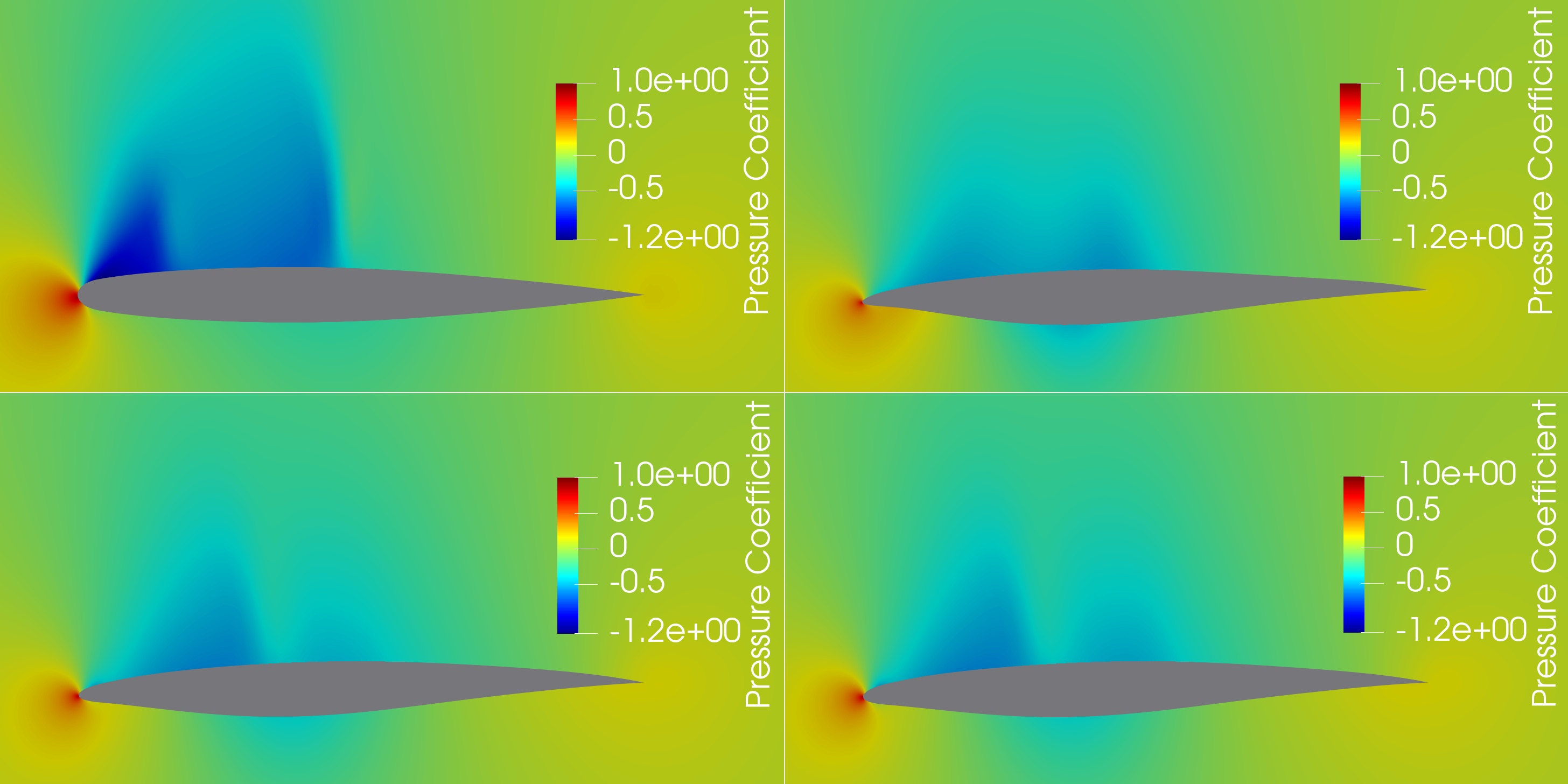}
 \caption{Cross-sections for optimized geometries. Original design (upper left), result SLSQP (upper right), result gradient descent (lower left), result Sobolev smoothing (lower right)}
 \label{picONERAM6cscomp}
\end{figure}

A second visualization of the computed designs is shown in Figure \ref{picONERAM6cscomp}, namely wing cross-sections at $65 \%$ of the length, together with the surrounding flow fields. This gives an impression of the thickness and surface curvature for the different optimized designs. We can observe that the SLSQP algorithm results in a thinner leading edge of the wing, while the other optimizations result in rounder leading edges. This is a sign of the different methods finding different local minima, a common phenomenon in nonlinear optimization. \\
Overall, we can observe that the application of our optimization method, with smoothing as a Hessian approximation, yields a more regular search direction and a faster rate of convergence for the optimizer if compared to gradient descent. Here, we can observe the effect of changing from a $\mathrm{L}^{2}$ to a $\mathrm{H}^{1}$ gradient, i.e., a reinterpretation of the derivatives in a different scalar product.


\begin{table}[H]
\centering
 \begin{tabular}{| c | c c c c |}
  \hline
   & RSQP surface & RSQP volume & SLSQP & grad. desc. \\
  \hline
   time (sec) & $2270166.91$ & $263644.04$ & $173046.68$ & $283626.55$ \\
  \hline
   time per step (sec) & $10806.68$ & $10545.76$ & $9613.7$ & $11345.06$ \\
  \hline
 \end{tabular}
 \caption{Averaged time consumption for different optimizers.}
 \label{tableRSQPtimeandmem}
\end{table}

\begin{table}[H]
\centering
 \begin{tabular}{| c | c c |}
  \hline
   & One Shot preconditioned & One Shot grad. desc. \\
  \hline
   time skylake (sec) & $41280.76$ & $39346.5$ \\
  \hline
 \end{tabular}
 \caption{Averaged time consumption for different One Shot optimizers.}
 \label{tableOneShottimeandmem}
\end{table}

Another important factor to take into account is the runtime necessary to achieve the presented optimization results. The computational costs of the Hessian computation are negligible compared to the time needed to solve the flow and (multiple) adjoint simulations, as we will see now. For this, we run all algorithms on the same setting on a high-performance computing (HPC) cluster. The test architecture are two Intel skylake XEON SP 6126 nodes with $12$ cores and $96$ GB of RAM each, where the simulations are executed on $48$ parallel MPI processes. We perform a number of $25$ optimizer steps, or less if convergence is reached before. \\
The measured runtimes can be seen in Table \ref{tableRSQPtimeandmem}. The SciPy optimizer excels in this setting since it requires only $18$ steps to converge and is thus much faster in overall time. Also applying the reduced SQP method with surface smoothing shows are very competitive performance in terms of runtime. It should be also taken into account that the flow and adjoint solutions are used as restart values in the next optimization. Therefore, the extra cost for the Hessian computation for the reduced SQP algorithm, compared to gradient descent, can be completely offset by smoother deformations leading to faster convergence in the next step. \\


Next, we assess the performance of our preconditioning method when applied in the context of One Shot optimization. Here, the hybrid Laplace-Beltrami operator from Equation \eqref{hybridlaplacebeltrami} takes the place of the Hessian preconditioner in the One Shot algorithm \ref{algoneshotdircon}. \\
Before we present the results there are some settings for the algorithm to be considered. First, we must choose an adequate number of inner piggyback steps between design updates. Obviously, this will depend on the exact test case, to ensure that the flow and adjoints are recovered after each design update. For the presented ONERA M6 test case using $10$ piggyback steps shows the desired behavior. Linked to this is the maximum allowed update step size. While the preconditioning itself enforces a scaling on the gradient, numerical experiments have shown the importance of further restrictions. Together with the number of piggyback steps, this choice helps to achieve a good combination of adjoint stability and design progress. \\
We will assemble the Laplace-Beltrami operator on the surface mesh, since it is computationally cheap and showed better performance with converged gradients. As weights we use the same setting as previously, i.e., $\epsilon_{1}= 56.9$, $\epsilon_{2}= 0.9$, and $\epsilon_{3}=0.1$. \\
Once again we compare our approach to gradient descent and SLSQP. The gradient descent algorithm in the context of One Shot needs preconditioning as well to remain stable. Here, we restrict the max step size, which is equivalent to using a multiple of the identity as a preconditioner. The SLSQP algorithm is easy to set up as well, since it only requires handing the function values and gradients computed by piggyback to the SciPy interface.


\begin{figure}[ht]
 \center
 \includegraphics[width=\linewidth]{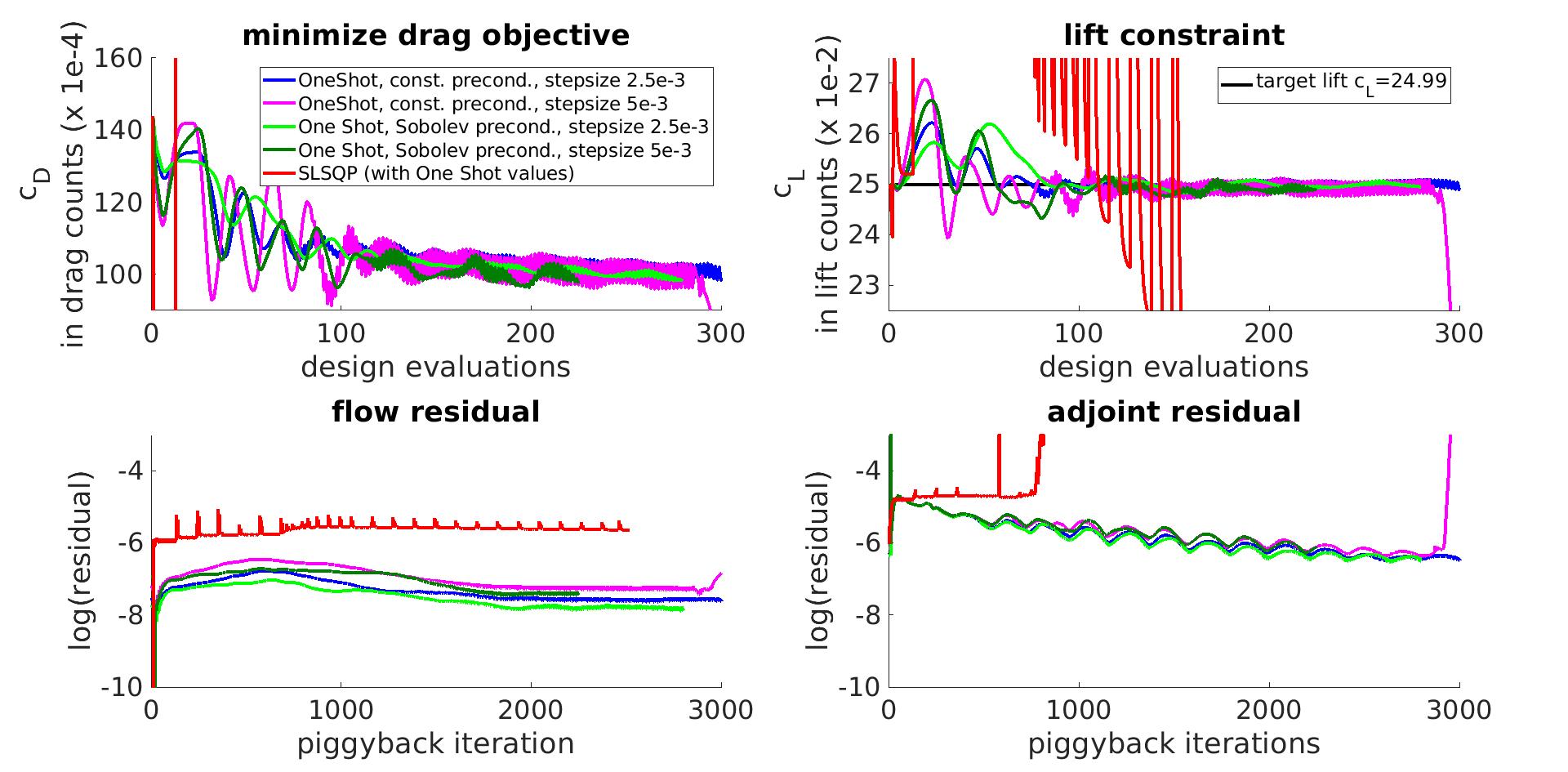}
 \caption{Comparison of different One Shot optimizations for the ONERA M6 test case}
 \label{picONERAM6optcomp_os}
\end{figure}

We can see the results for One Shot optimization in Figure \ref{picONERAM6optcomp_os}. From the plot, we can draw several observations. First, our Hessian approximation, shown in dark green, performs a successful optimization and achieves a minimal drag coefficient as low as $100.5$ drag counts, with our proposed smoothing preconditioner. This represents a $30.14\%$ improvement, while having $c_{L}=24.985$ lift counts, which is ca. $0.02\%$ deviation from the target lift. Here, two settings for the steep sizes are shown. A heuristic is used to limit the maximal design update to the values given in Figure \ref{picONERAM6optcomp_os}. As can be seen by the light green line, limiting the value from $5e^{-3}$ to $2.5e^{-3}$ slows the algorithm, but gives a similar result. \\
Second, the gradient descent method, shown in blue, also performs a successful optimization, in accordance with the theory stating that for small enough update steps the One Shot algorithm must be convergent. The optimization results in an average drag coefficient of $101.6$ drag counts, or $29.38\%$ improvement, while showing small oscillating behavior in the end. Comparing the dark green and blue lines in Figure \ref{picONERAM6optcomp_os}, it can be observed that constant preconditioners lead to higher oscillations for equal step sizes. Also, they tend to diverge faster if the maximal design updates are not limited carefully, as shown by the magenta line in Figure \ref{picONERAM6optcomp_os}. This means that the use of an approximated, reduced shape Hessian as a preconditioner will allow the One Shot algorithm to perform larger design updates and result in a faster convergence to an optimal design. While at the same time keeping the additional time to assemble the Sobolev smoothing operator low, as can be seen from the runtimes measured in Table \ref{tableOneShottimeandmem}. \\
The third observation is that a direct application of the SLSQP algorithm diverges after a few optimization steps. This can be explained as well since it is well known that BFGS updates will suffer from a high number of matrix updates. A problem, which is further increased by using approximated One Shot gradients of questionable precision. \\
Two more general points to notice are how convergent One Shot optimizations adhere to all design constraints and how in general One Shot needs more optimization iterations than working with converged values. This is natural since we intend to do more, and in tendency smaller, design updates, albeit each at a fraction of the price for computing exact function and gradient values.


\begin{figure}[ht]
 \center
 \includegraphics[width=\linewidth]{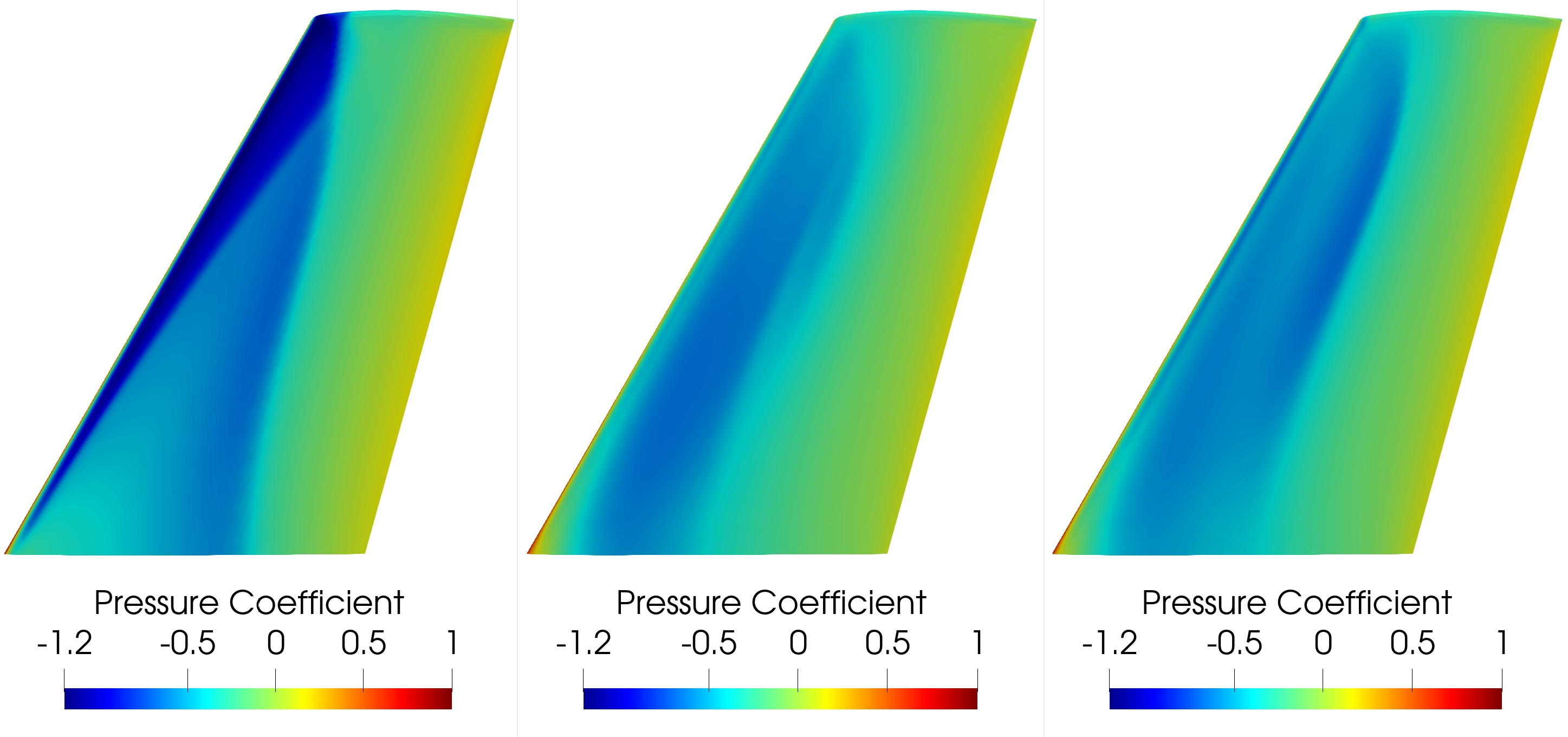}
 \caption{Pressure coefficient on the wing surface for optimized solutions. Original design (left), result Sobolev smoothing (center), result gradient descent (right)}
 \label{picOM6OScp}
\end{figure}

\begin{figure}[ht]
 \center
 \includegraphics[width=\linewidth]{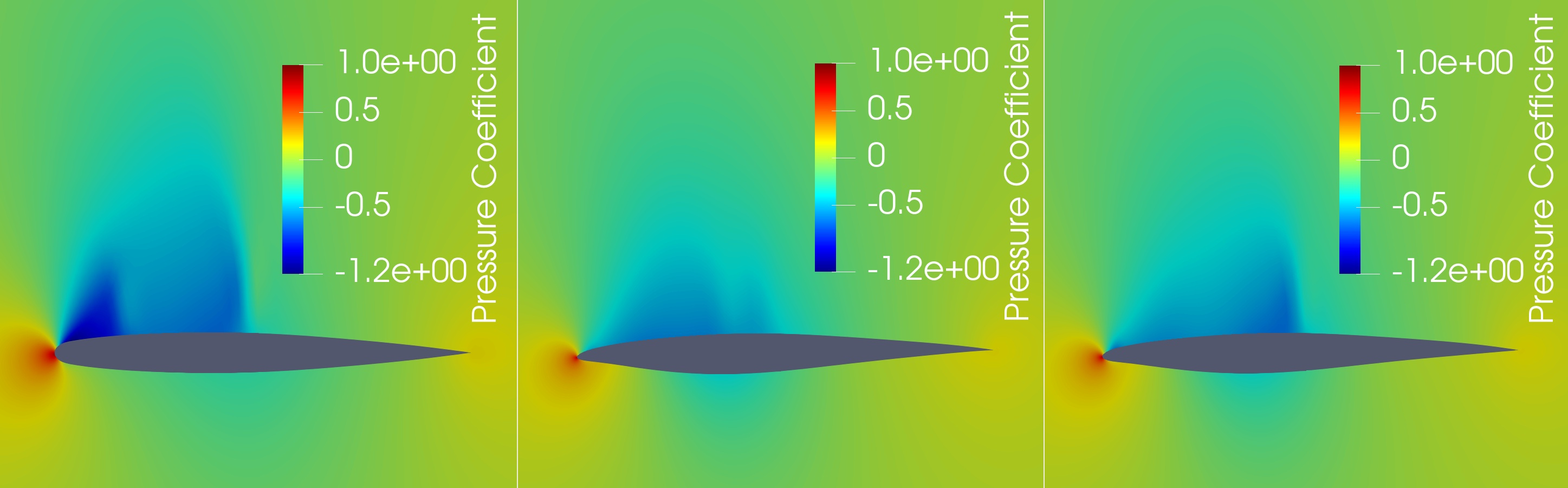}
 \caption{Cross section at $\frac{y}{l}=0.65$ on the wing for optimized solutions. Original design (left), result Sobolev smoothing (center), result gradient descent (right)}
 \label{picOM6OScrosssect}
\end{figure}

After considering the optimization progress, we will investigate the resulting wing profiles in more detail. In Figure \ref{picOM6OScp}, the pressure coefficient on the surface is shown in relation to the original distribution. As we can observe, the One Shot method efficiently removes the characteristic lambda shock. Both optimizers remove the shock, thereby reducing the induced drag from the system. In Figure \ref{picOM6OScrosssect}, we can see the cross-section at $65\%$ of the wing length. Here, we can observe how the double shock gets reduced, and it is also worth noting how the final cross-sections are very close to the one found by the SLSQP algorithm, with converged values, from Figure \ref{picONERAM6cscomp}.


Next, we take a look at the runtime of the One Shot algorithm. All computations are done on an HPC cluster and we use the same setting as in the previous runtime test. The measured times can be found in Table \ref{tableOneShottimeandmem}. For this test, we perform $250$ One Shot design updates with each algorithm. As we can see, the use of a One Shot approach leads to significant improvements in runtime. This proves that performing a higher number of cheap design updates can be advantageous compared to doing a few costly steps. \\
To better summarize this point, we can calculate the so-called retardation factor. That is the quotient between the time for the complete optimization and one full convergence of the flow solver. Since converging the flow simulation requires $2224.1$ sec. on the two skylake nodes, we end up with the factors shown in Table \ref{tableretardfact}. The numbers presented here clearly show that applying One Shot can give crucial advantages in runtime. The presented Hessian approximation techniques can be used in this context as efficient preconditioning to ensure fast and stable convergence of the One Shot algorithm. All while keeping the computational cost low. \\
If we compare the time retardation factors, we can conclude that the One Shot optimization algorithm achieves a full optimization in between $18-20$ times the wall clock time needed for a simple flow simulation. At the same time, when working with exact gradients, even the competitive SLSQP implementation requires ca. $77$ times the runtime at best. The retardation factors reveal that One Shot optimization with our proposed Sobolev smoothing is between $3.98-4.19$ times more efficient, in terms of runtime than established Quasi-Newton methods with exact derivatives. \\
For One Shot optimizers it is common practice to calculate retardation factors w.r.t. the iteration count, too. Here, the base value are the $15371$ iterations necessary to fully converge the flow solver. For the optimizer, we need to converge the flow and adjoint solvers to a certain level, in order to gain initial restart solutions. Then for the given setting the One Shot optimization must perform $1$ flow and $2$ adjoint steps during each piggyback iteration and at the end, the flow is fully converged again. This results in the iteration retardation factors stated in Table \ref{tableretardfact}. The discrepancy between the time and iteration retardation is largely due to the additional restarts and file I/O our implementation has to perform. Overall, we can conclude that the presented smoothing techniques are an effective preconditioning method for One Shot algorithms and result in a highly effective optimization strategy.

\begin{table}[ht]
\centering
 \begin{tabular}{| c | c c |}
  \hline
   & time retardation & iteration retardation\\
  \hline
   RSQP surface & $121.47$ & - \\
  \hline
   RSQP volume & $118.54$ & - \\
  \hline
   SLSQP & $77.81$ & - \\
  \hhline{|=|==|}
   One Shot surface & $18.56$ & $4.08$ \\
  \hline
   One Shot grad. desc. & $19.53$ & $4.17$ \\
  \hline
 \end{tabular}
 \caption{Retardation factors for optimization}
 \label{tableretardfact}
\end{table}

\section{Conclusion}
\label{secconc}

In this paper, we have shown how Sobolev smoothing further improves performance in aerodynamic shape optimization, even in situations where a smooth parameterization is used anyway. This is achieved by interpreting Sobolev smoothing as an approximation of the reduced shape Hessian, which is then incorporated into the parameterized setting via a tailor-made adaptation of the Faà di Bruno formula. \\
Although valid for all kinds of parameterizations, the resulting expressions based on our modified Faà di Bruno formula are particularly easy to implement for linear parameterizations. This makes the novel methodology presented herein convenient for a wide variety of industrial and engineering problems, since most of the common parameterizations, such as Hicks-Henne functions or FFD boxes, are typically linear. \\
Furthermore, our modified Faà di Bruno formula can immediately be used as the KKT system of an approximated Newton-type optimizer, such as One Shot. The Laplace-Beltrami operator for Sobolev smoothing is extended by coefficients for zeroth and second-order smoothing. We compare this One Shot optimizer with a state-of-the-art SLSQP optimizer for industry-sized problems. \\
This is demonstrated by a shape optimization of an  ONERA M6 wing for drag reduction, subject to constant lift and thickness. For comparable gain in drag reduction, we achieve a time retardation factor of $18.56$ with our One Shot approach, compared to $77.81$ with the SLSQP optimizer. This represents a decrease factor of $4.19$ in terms of runtime.

\section{Acknowledgments}
\label{secacknowledge}

The authors acknowledge the support by Lisa Kusch concerning discussions on One Shot optimization and by Payam Dehpanah (both Scientific Computing Group, TU Kaiserslautern) for creating computational meshes for the numerical test cases. In addition, the authors acknowledge the productive discussions with Bernhard Eisfeld (DLR Braunschweig) and his valuable feedback on the ONERA M6 test case used in this paper. The presented computations were performed with resources provided by the Computing Center (RHRK) at TU Kaiserslautern.

\section*{Appendix A: Implementation}
\label{secAppendix}

The methods discussed in this paper were implemented within the SU2 solver \cite{SU2016}. SU2 is an open-source multiphysics simulation and design software. For the scope of our work, the most important components are the available design optimization capabilities and the included discrete adjoint solver \cite{Albring15}. The adjoint solver is based on algorithmic differentiation, via operator overloading, and ensures that the adjoint solver is consistent for the governing flow and turbulence equations. \\
We implement an efficient Sobolev smoothing approach, as introduced in previous sections, by providing additional functionality in two main areas.
\begin{enumerate}
 \item An extension of the discrete adjoint flow solver to enable the assembly, solution, and output of the approximated Hessian matrix. For this, the Laplace-Beltrami operator is constructed on the mesh level and multiplications with the parameterization Jacobian are calculated.
 \item The addition of an SQP optimizer capable of working with approximated Hessian matrices. Together with the existing design capabilities and the new Hessian approximation, this allows for efficient optimization. These capabilities are added to the existing python optimization toolbox FADO.
\end{enumerate}

\subsection*{Gradient Smoothing Solver}

The core of the presented implementation is a gradient smoothing solver capable of constructing the Hessian approximation from Equation \eqref{systemmatrixequation} and handing it to an optimizer. To explain the key ideas of the implementation first, we need to introduce the basic structure of the SU2 C++ code. The current layout is based on a hierarchy of classes, which fulfill different roles for the multiphysics simulation.  \\
\begin{figure}[H]
 \center
 \includegraphics[width=\linewidth]{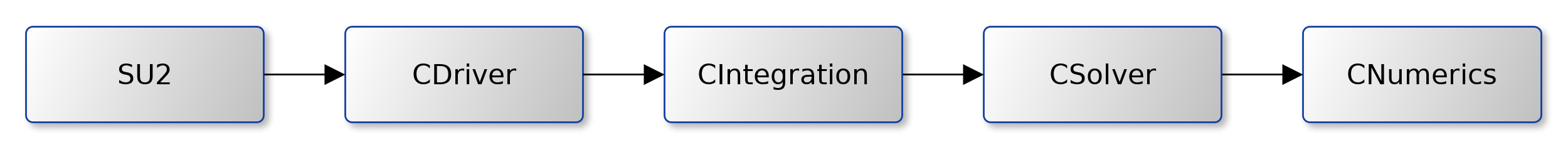}
 \caption{Class hierarchy in SU2}
 \label{picSU2classhierachy}
\end{figure}
The general class hierarchy of SU2 is shown in Figure \ref{picSU2classhierachy}. For our purposes the two most important levels are:
\begin{enumerate}
 \item \textbf{driver classes:} These are high-level classes derived from a common interface named 'CDriver'. They determine the kind of physical problem and options can include whether this is a flow, structural, or heat transport problem, or an adjoint run. They also handle the spatial layout of the problem being all on a single zone, e.g., only flow solutions, or having multiple zones, e.g., fluid-structure interactions. Finally, they control the temporal nature of the involved physics, which could be a steady-state problem with pseudo-time-stepping or a time-dependent problem. Here, we are interested in the 'CDiscAdjSingelzoneDriver' class, which runs the discrete adjoint simulation of a flow problem, thus computing the sensitivities of an objective function with respect to the mesh nodes.
 \item \textbf{solver classes:} These classes implement a specific numerical solver and represent a certain set of governing equations. This could be flow, turbulence, structure, or any other kind of partial differential equation. All solver classes are based on a common interface named 'CSolver'. For our purposes, we are interested in performing operations on the solution of the adjoint flow solver class 'CDiscAdjSolver'.
\end{enumerate}
Considering the structure, it is clear that to implement a Sobolev gradient treatment and Hessian approximation, we need to adapt the existing discrete adjoint framework on those two levels. \\
First, we implement a new solver class 'CGradientSmoothingSolver', which has to assemble and solve the discretized PDE from Equation \eqref{completesmoothingequation}. From the structure of the Laplace-Beltrami operator, it is obvious that finite elements can be used to get an efficient discrete representation. This is done by extending the existing finite element implementation in SU2, provided by the structural mechanic features of the code. Several new functions are necessary to overcome two main hurdles. First, we need to extend the numerical integration to use second-order Gauss quadrature. This is necessary to accurately represent the identity with finite elements. Second, the existing elements and transformations are extended to deal with surface cells. For example, a two-dimensional surface cell is embedded in a three-dimensional mesh, meaning the transformation to reference elements maps $\R^{3} \rightarrow \R^{2}$ and this change in dimension requires some extensions to the existing implementation. Overall, this solver assembles the discrete equivalent of the Laplace-Beltrami operator on the computational mesh. \\
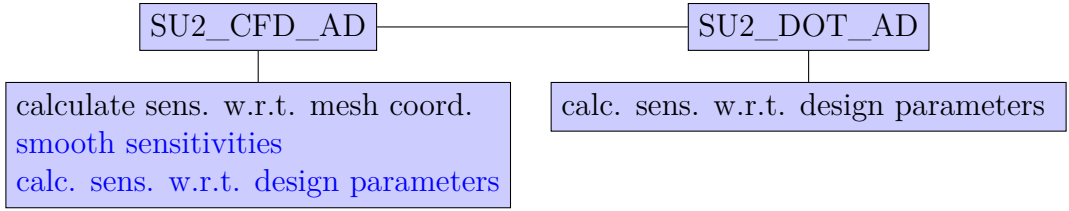
\begin{figure}[H]
 \begin{center}
  \begin{tikzpicture}[node distance = 1.5cm, auto]
   \node (driver) [Base] {SU2\_CFD\_AD}
    child { node (drivers) [Base, below=1em of driver] {calculate sens. w.r.t. mesh coord. \\
    \textcolor{blue}{smooth sensitivities} \\
    \textcolor{blue}{calc. sens. w.r.t. design parameters} } }
    child { node (solver) [Base, right=10em of driver] {SU2\_DOT\_AD}
      child { node (solvers) [Base, below=1em of solver] {calc. sens. w.r.t. design parameters } } };
  \end{tikzpicture}
 \end{center}
 \caption{Different executables and their capabilities in SU2}
 \label{picdifferentexe}
\end{figure}
Next, the parameterization Jacobian $\D{p} M(p)$ and its transposed must be available in some form. Normally, all driver and solver classes exist within an executable SU2\_CFD\_AD and the 'CDiscAdjSinglezoneDriver' only computes sensitivities with respect to the mesh coordinates. The chain rule evaluation to get sensitivities w.r.t. design parameters is then handled by a different executable named SU2\_DOT\_AD, as is depicted in Figure \ref{picdifferentexe}. This modular setup allows the adjoint solvers to be independent of the choice of parameterization, yet for our implementation, it is necessary to transfer functions for the projection of sensitivities from SU2\_DOT\_AD.
Our implementation enables the computations of all the steps inside of a single instance of the class 'CGradientSmoothingSolver', including the computation of the Jacobian matrix of the mesh parameterization. Reviewing Equation \eqref{completesmoothingequation} carefully shows that the full matrices are not needed at all.Consider the involved dimensions, $\D{p} M(p) \in \R^{(d \cdot n_m) \times n_{p}}$, $(\epsilon_1 \mathrm{I} - \epsilon_2 \triangle) \in \R^{(d \cdot n_m) \times (d \cdot n_m)}$, and $B \in \R^{n_{p} \times n_{p}}$. We expect $n_{p} << n_{m}$ and therefore it is much more efficient to only construct the complete matrix $B$. If we have a way to compute arbitrary scalar products with $\D{p} M(p)$ and $\D{p} M(p)^{T}$ this can be achieved. \textit{Algorithmic Differentiation} (AD) is an efficient way to calculate derivatives \cite{Griewank08}, and is especially efficient if only matrix-vector products with Jacobians are required. AD offers two basic options called the forward and reverse mode. Given a function $M: \R^{n_p} \rightarrow \R^{d \cdot n_m}$ and arbitrary vectors $\dot{x} \in \R^{n_p}, \bar{y} \in \R^{d \cdot n_m}$ we can compute
\begin{equation}
\tag{A.1}\label{eqAD}
\begin{aligned}
 \dot{y} & = \D{p} M(p) \dot{x} & \text{(forward mode)} \\
 \bar{x} & = \D{p} M(p)^{T} \bar{y} & \text{(reverse mode).}
\end{aligned}
\end{equation}
In our implementation, we use the AD tool CoDiPack \cite{Sagebaum19} to efficiently calculate the derivatives. CoDiPack is already the standard AD tool in SU2 and is heavily utilized to solve the adjoint equation. Here, we use the available CoDiPack features to efficiently evaluate the two expressions in Equation \eqref{eqAD}. With all these functions in place, we can assemble and solve Equation \eqref{completesmoothingequation} in SU2, as well as extract the left-hand side matrix for use as a Hessian approximation. \\
Finally, the existing 'CDiscAdjSinglezoneDriver' class is extended in several ways. We add an instance of the new 'CGradientSmoothingSolver', and couple it with the discrete adjoint solver as a postprocessing step. Together, they enable us to calculate $D_{p} L$ and $B$ with one adjoint simulation run. \\

\subsection*{Piggyback Driver}

The One Shot algorithm \ref{algoneshotdircon} is built around the idea of a piggyback iteration. This means we compute a coupled iteration of flow and adjoints, where the current flow state $u_{j}$ is used to compute the current adjoint state $\lambda_{j}$. While we will not go into specific details about the implementation, it is worth pointing out the key idea and how this can be set up. \\
Within SU2 the general adjoint logic is included within the driver classes, derived from a common interface 'CDriver', see Figure \ref{picSU2classhierachy} for the class hierarchy. Driver classes represent a certain type of physical problem, e.g., flow, structure, adjoint, etc., and the driver controls the time-stepping process and calls the involved solvers. \\
For many adjoint solvers, including the SU2 implementation, adjoints are computed using the idea of reverse accumulation. The algorithm is based on a theorem by Christianson regarding AD of fixed-point iterations \cite{Christianson98}. Usually, when applying the reverse mode of AD to an iterative procedure, e.g., $u_{j} = G(u_{j},m)$, one needs to store all intermediate states $u_{1},\dots,u_{J}$ for the derivative computation. Of course, this results in unfeasible memory overheads. The theorem states that for a converged fixed point solution $u_{*}$ it is sufficient to store only this fixed point to compute the exact derivatives.

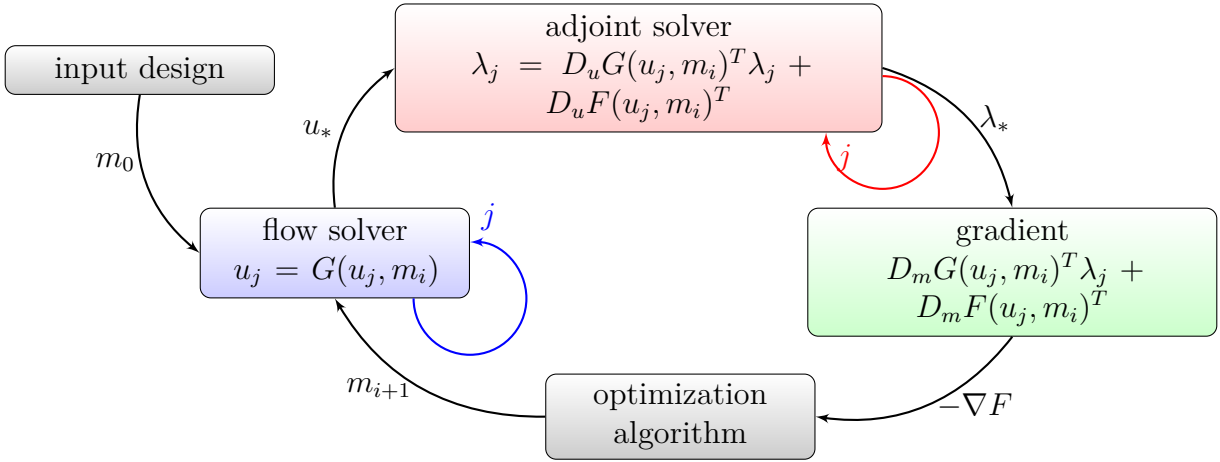
\begin{figure}[ht]
 \begin{center}
  \tikzstyle{block} = [rectangle, draw,text centered, minimum height=0.5em, top color=white,bottom color=black!20]
  \tikzstyle{line} = [draw, -latex']
  \begin{tikzpicture}
     \node[block, text width=8em, rounded corners, top color=white,bottom color=blue!20] (Primal) {flow solver\\$u_{j} = G(u_{j}, m_{i})$};
     \node[block, text width=15em, rounded corners, above right = 1cm and -1cm of Primal, top color=white,bottom color=red!20] (Adjoint) {adjoint solver\\$\lambda_{j} = D_{u} G(u_{j}, m_{i})^{T} \lambda_{j} + D_{u} F(u_{j}, m_{i})^{T}$};
     \node[block, text width=12.5em, rounded corners, below right = 1cm and -1cm of Adjoint, top color=white,bottom color=green!20] (Gradient) {gradient\\$D_{m} G(u_{j}, m_{i})^{T} \lambda_{j} + D_{m} F(u_{j}, m_{i})^{T}$};
     \node[block, text width=8em, rounded corners, below right = 1cm and 1cm of Primal, top color=white,bottom color=black!20] (Opt) {optimization\\algorithm};
     \node[block, text width=8em, rounded corners, above left = 1.5cm and -1.0cm of Primal, top color=white,bottom color=black!20] (Start) {input design};
     \draw[line, thick, black, bend left] (Primal.north) to node[left] {$u_{*}$} (Adjoint.west);
     \draw[line, thick, black, bend left] (Adjoint.east) to node[right] {$\lambda_{*}$} (Gradient.north);
     \draw[line, thick, black, bend left] (Gradient.south) to node[right] {$-\nabla F$} (Opt.east);
     \draw[line, thick, black, bend left] (Opt.west) to node[left] {$m_{i+1}$} (Primal.south);
     \draw[line, thick, black, bend right] (Start.south) to node[above left] {$m_0$} (Primal.west);
     \draw [line, thick, blue] (Primal.south east) ++(-0.75cm,0pt) arc [start angle=-180, end angle=90, radius=0.75cm] node[above right]{$j$};
     \draw [line, thick, red] (Adjoint.south east) ++(0pt,+0.75cm) arc [start angle=90, end angle=-180, radius=0.75cm] node[below right]{$j$};
  \end{tikzpicture}
 \end{center}
 \caption{Optimization process with reverse accumulation}
 \label{figOptRA}
\end{figure}

In terms of data flow, the situation is depicted in Figure \ref{figOptRA}. First, an instance of the flow driver executes a pseudo-time-stepping to solve the flow equation, then the resulting solution $u_{*}$ is handed to the adjoint driver. There an iteration is set up using reverse-mode AD and then executed until a converged adjoint state $\lambda_{*}$ is reached. \\
For our implementation, we introduced a new piggyback driver class named 'COneShotSinglezoneDriver'. Following the idea presented in Equation \eqref{eqpiggyback} the data flow is changed to the layout depicted in Figure \ref{figOptOS}.

\begin{figure}[ht]
 \begin{center}
  \tikzstyle{block} = [rectangle, draw, text centered, minimum height=0.5em, top color=white,bottom color=black!20]
  \tikzstyle{line} = [draw, -latex']
  \begin{tikzpicture}
     \node[block, text width=10em, rounded corners, node distance=4cm,top color=white,bottom color=blue!20] (Piggy) {flow solver\\$u_{j} = G(u_{j}, m_{i})$\\adjoint solver\\$D_{u} N(u_{j}, \lambda_{j}, m_{i})^{T} = 0$};
     \node[block, text width=8em, rounded corners, right = 4cm of Piggy, top color=white,bottom color=green!20] (Gradient) {gradient\\$D_{m} N(u_{j}, \lambda_{j}, m_{i})$};
     \node[block, text width=8em, rounded corners, below right = 1.5cm and 1cm of Primal, top color=white,bottom color=black!20] (Opt) {optimization\\algorithm};
     \node[block, text width=8em, rounded corners, above left = 1.5cm and -0.75cm of Primal, top color=white,bottom color=black!20] (Start) {input design};
     \draw[line, thick, black, bend left] (Piggy.east) to node[above] {$u_{*}, \lambda_{*}$} (Gradient.west);
     \draw[line, thick, black, bend left] (Gradient.south) to node[right] {$-\nabla F$} (Opt.east);
     \draw[line, thick, black, bend left] (Opt.west) to node[left] {$m_{i+1}$} (Piggy.south);
     \draw[line, thick, black, bend right] (Start.south) to node[above left] {$m_0$} (Piggy.west);
     \draw [line, thick, blue] (Piggy.south east) ++(-0.75cm,0pt) arc [start angle=-180, end angle=90, radius=0.75cm] node[above right]{$j$};
  \end{tikzpicture}
 \end{center}
 \caption{Optimization process with piggyback}
 \label{figOptOS}
\end{figure}
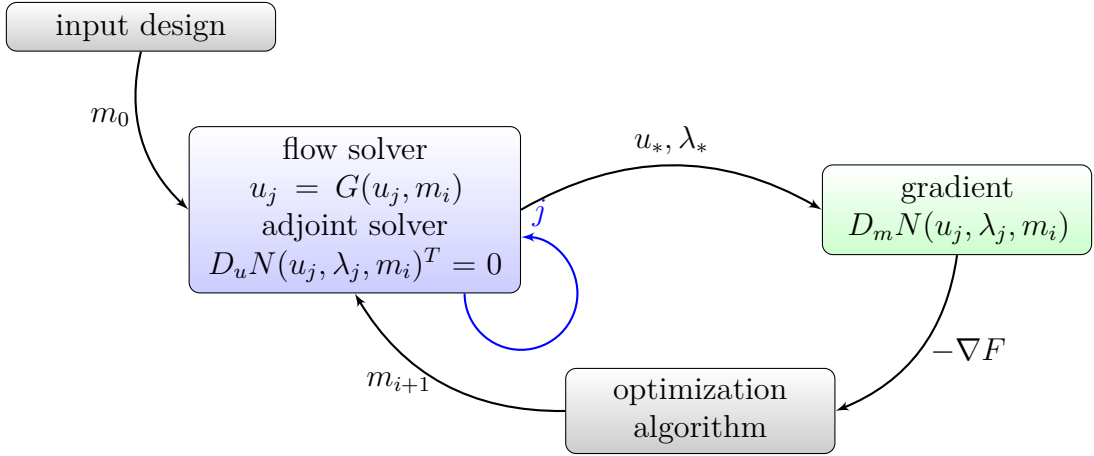

The new driver performs a piggyback iteration, by doing one step of the flow solver and then handing the intermediate result $u_{j}$ to the adjoint solver to do one adjoint step using reverse-mode AD, and returns $\lambda_{j}$. Those two steps are then repeated to perform a combined fixed point iteration. \\
Using these capabilities, we are able to compute all necessary reduced gradients $\Dred{p} F, \Dred{p} E_{k}, \Dred{p} C_{l}$ from Algorithm \ref{algoneshotdircon} with the SU2 adjoint framework. Please keep in mind that these principles apply to all AD-based adjoint solvers when used for One Shot optimization. By using a consistent piggyback implementation we ensure that function and gradient values, while inaccurate, are still computed consistently with one another.

\subsection{Optimizer implementation}

In Section \ref{secoptalg}, we introduced an SQP optimization strategy using Sobolev gradient smoothing for design parameters to compute an approximated Hessian. Here, we will discuss the implementation of Algorithms \ref{eqconstSQP}, \ref{constSQP}, and \ref{algoneshotdircon}. \\
The flow and adjoint simulations in this paper are done using the SU2 framework. SU2 already provides a number of python scripts to simplify various tasks, by a python interface. These scripts allow the computation of flow and adjoint simulations, using the different SU2 C++ executables, and the evaluation of functionals and gradients, respectively. The results can be handed to the SciPy optimization library \cite{2020SciPy-NMeth} by the interface. However, the provided optimization capabilities are insufficient to run Algorithms \ref{constSQP} and \ref{algoneshotdircon} for several reasons. The major concern is that we want to provide an approximated Hessian to the SQP optimizer, while most software packages either expect an exact Hessian or numerically approximate the Hessian themselves. \\
To be able to provide our Hessian approximation we decided to implement the core iterative loop of the SQP method ourselves and use existing packages only to solve the individual steps of the optimization algorithms. To achieve this we use the FADO toolbox\footnote{FADO - Framework for Aerostructural Design Optimization, \url{https://github.com/pcarruscag/FADO}}. This toolbox provides a python interface to call SU2 simulations, to store and control different designs and meshes, or to adapt config options. It is deliberately open for the addition of different optimizers and we added an implementation of the basic SQP algorithm to run our process. In addition, in Algorithms \ref{constSQP} and \ref{algoneshotdircon} we have to deal with a more complex constrained quadratic program. For this, we use the CVXOPT python package\footnote{CVXOPT - Python Software for Convex Optimization, \url{https://cvxopt.org}} for convex programming \cite{cvxopt10, cvxopt11}. The Hessian approximation in Equation \eqref{systemmatrixequation} is symmetric, positive definite, which implies that the quadratic subproblems for SQP are convex, so this is a valid choice. \\

 \bibliographystyle{elsarticle-num}
 \bibliography{Bibliography/Bibliography.bib}

\end{document}